\documentclass[twocolumn]{autart}    
\usepackage{amsmath}
\usepackage{amsfonts,amssymb}
\usepackage{CJK}
\usepackage{multicol}
\usepackage{multirow}
\usepackage{diagbox}
\usepackage{graphicx}
\usepackage{subfig} 
\usepackage{lineno} 
\allowdisplaybreaks 
\usepackage{epstopdf} 
\usepackage{pgf,tikz}
\usepackage{mathrsfs}
\usepackage{tikz}
\usepackage{graphicx}     
\usepackage[utf8]{inputenc}
\usepackage{subfig} 
\usepackage{subfloat}
\usepackage{xcolor}
\usepackage{colortbl,booktabs}
\usepackage{cite}
\usepackage{graphicx}
\usepackage{subfig}
\usepackage{algorithm}
\usepackage{algorithmicx}
\usepackage{algpseudocode}

\allowdisplaybreaks[4]
\newtheorem{theorem}{Theorem}
\newtheorem{lemma}{Lemma}

\newtheorem{remark}{Remark}
\newtheorem{assumption}{Assumption}

\definecolor{blue}{rgb}{0,0,0}

\newbool{submitA} \setbool{submitA}{false} 

\usepackage[round]{natbib}

\begin{document}
	
	\begin{frontmatter}
		\title{Distributed Seeking for Fixed Points of Biased Stochastic Operators: A Communication-Efficient Approach} 
		
		\thanks[footnoteinfo]{This work was supported in part by the National Natural Science Foundation of China under Grant 62133003.}
	
		\thanks[Corresponding author]{Corresponding author.}
		
		\author[111]{Fan Li}\ead{funnynice2024@163.com},
		\author[222]{Lei Xu}\ead{lei5@kth.se},
		\author[333]{Xinlei Yi}\ead{xinleiyi@tongji.edu.cn},
		\author[444]{Guanghui Wen}\ead{wenguanghui@gmail.com},
		\author[555]{Yang Shi}\ead{yshi@uvic.ca},
		\author[111,Corresponding author]{Tao Yang}\ead{yangtao@mail.neu.edu.cn},  

		\address[111]{State Key Laboratory of Synthetical Automation for Process Industries, Northeastern University, Shenyang 110819, China}  
		\address[222]{Division of Decision and Control Systems, School of Electrical Engineering and Computer Science, KTH Royal Institute of Technology, and also affiliated with Digital Futures, 100 44, Stockholm, Sweden}
		\address[333]{Department of Control Science and Engineering, College of Electronics and Information Engineering, Tongji University, China}
		\address[444]{Department of Systems Science, School of Mathematics, Southeast University,  China}
		\address[555]{Department of Mechanical Engineering, University of
			Victoria, Canada}


\begin{abstract}                          
This paper investigates the distributed fixed point seeking problem of sum-separable stochastic operators over the multi-agent network. Based on inexact Krasnosel'ski\u{\i}--Mann  iterations, the communication-efficient distributed algorithm is proposed under the relaxed growth bias and variance conditions, generalizing traditional unbiased and bounded additive variance assumptions. To enhance communication efficiency, we integrate communication compression and dynamic period skipping techniques, particularly adopting a unified compressor that allows both relative and absolute compression errors. By introducing a surrogate function for general non-contractive and contractive operators, we establish convergence guarantees of the distributed fixed point iteration, achieving among the first theoretical unifications with distributed non-convex optimization algorithms.   Finally, numerical simulations validate the effectiveness of the theoretical results.
\end{abstract}
\end{frontmatter}
	
		\vspace{-3mm}
	\section{Introduction}
		\vspace{-3mm}
	Fixed point theory, being fundamental to optimization analysis~\citep{book1}, offers a rigorous framework by interpreting iterative algorithms as methods for seeking fixed points of associated operators. This operator-theoretic approach not only unifies convergence analysis but also guides novel algorithmic designs~\citep{OptTAC}. However, centralized fixed-point methods face scalability bottlenecks in large-scale networks, prompting the shift to distributed settings. Such distributed frameworks are crucial in machine learning~\citep{Icml2014}, network control~\citep{Tang}, and large-scale optimization~\citep{Yang2019,XinleiAuto} due to their inherent scalability and broad applicability.
	
		\vspace{-3mm}
	
	Centralized fixed point research has progressed from Picard iteration to Krasnosel'ski\u{\i}--Mann (KM) iterations~\citep{KM1955} for non-expansive operators. Yet, these methods struggle with cost and scalability in network applications. In response, distributed fixed-point iterations without a global coordinator have gained focus. For example,~\cite{XiuTAC,for3,for2} explore common fixed points for operator families but require restrictive assumptions on the existence common fixed points for all local operators, hindering application to distributed optimization. To overcome this,~\cite{XiuAuto} pioneer the sum-separable operator case via distributed KM iterations. Subsequent studies~\cite{locTac,DOT} develop linearly convergent distributed algorithms for undirected and directed graphs, respectively, while~\cite{FanTcns} further extend these results to continuous-time dynamical systems. These advancements in distributed fixed point iteration provide powerful analytical tools for distributed optimization. Nevertheless, the above works invariably rely on the (quasi-)non-expansiveness of operators, severely limiting the generality of the theoretical framework. In order to recover non-convex optimization, it is necessary to focus on the case of general Lipschitz operators allowing expansiveness.
	
		\vspace{-3mm}
	
	The existing theoretical framework for distributed fixed point iterations is centered around deterministic operators. However, in practical scenarios, only noisy approximations of local operators are typically available. To broaden the applicability of operator theory,~\cite{NicolaTac} and~\cite{PointTac} investigate the fixed point seeking problem for stochastic operators based on inexact KM iterations, incorporating STORM-based variance-reduced variants and high-probability convergence guarantees, respectively. However, they are limited to centralized settings only.  
	Additionally, while research on distributed stochastic optimization algorithms is extensive, it largely relies on idealized assumptions of unbiased estimates. However, this premise is often violated in practical scenarios where mechanisms such as differential privacy or zeroth-order approximations inevitably induce systematic bias. Unlike zero-mean noise, such bias resists elimination via simple averaging and can severely degrade algorithmic convergence. Recently,~\cite{Pushi} extend this theoretical limit by relaxing the bounded additive variance condition to establish convergence under the relaxed growth variance condition for stochastic gradient tracking algorithms with non-convex objectives. Although~\cite{Pushi} relax the variance condition, it strictly relies on the unbiased gradient estimation assumption. Moving from unbiased to biased stochastic frameworks is essential for capturing practical sampling imperfections. How to design biased stochastic distributed optimization algorithms under mild conditions remains an open problem.

		\vspace{-3mm}
	
	As the network scale increases, the communication cost becomes a key bottleneck that restricts the efficiency of algorithms. Reducing the amount of information transmission and the frequency of communication constitutes two primary approaches to decreasing communication costs in distributed algorithms. Under limited communication bandwidth, compression mechanisms effectively reduce the amount of transmitted data. Building on biased compressors, the compression algorithm is designed in~\cite{ZhangjiaqiTAC} and achieves linear convergence by compressing the innovation error to avoid the possible divergence due to the accumulation of the compression error as reported in~\cite{Lei12} and~\cite{Lei13}. For non-convex objectives,~\cite{XinleiTAC} design distributed primal--dual algorithms and achieves linear convergence under the Polyak--Łojasiewicz condition. To encompass more types of quantizations in practical systems, one research direction is to generalize assumptions about compressors. Along this line,~\cite{Pushi} and~\cite{JOTA} have made efforts by introducing a unified compression assumption that allows both relative and absolute compression errors and established convergence guarantees. Complementary to communication compression, period skipping~\citep{TAC22} and local update strategies~\citep{LocalJ,LocalSun} have been widely adopted to diminish communication frequency by permitting multiple local iterations between network transmissions. However, these results are strictly confined to optimization problems. To date, no existing work has investigated communication-efficient distributed fixed point iterations for general operators. This absence is largely attributed to theoretical hurdles. Unlike optimization where objective function values provide a natural metric for convergence, fixed point iterations rely solely on operator properties. This makes controlling the accumulation of compression errors considerably more difficult, particularly for the non-contractive or expansive operators considered in this paper. Therefore, establishing a communication-efficient operator framework is both necessary and nontrivial.

		\vspace{-3mm}
	
	Based on the above observations, the aim of this paper is to design a unified framework of communication-efficient algorithms for distributed fixed point seeking of biased stochastic operators. In summary, the contributions can be summarized as follows.
	\begin{enumerate}
		
			\vspace{-3mm}
		
		\item  This paper investigates the distributed fixed point seeking problem for sum-separable biased stochastic operators, extending beyond the deterministic operator frameworks of prior studies~\citep{for3,for2,XiuAuto,DOT,locTac} and advancing the centralized stochastic setting in~\cite{PointTac} to a distributed environment. A key generalization lies in accommodating non-contractive operators that may exhibit expansiveness, thereby relaxing the non-expansiveness condition inherent in existing convergence analysis~\citep{FanTcns,XiuAuto,DOT}. Furthermore, we relax the traditional unbiased and bounded additive variance assumptions commonly used in distributed stochastic optimization~\citep{YuanAuto,Jueyou,TAC22},  establishing a novel convergence analysis framework under relaxed growth conditions for both bias and variance.

		\item Departing from the perfect communication settings in existing distributed operator frameworks~\citep{XiuAuto,DOT,FanTcns,locTac}, and building upon inexact KM iterations, we design a unified communication-efficient algorithm that simultaneously reduces both the volume and frequency of information transmission between agents by seamlessly integrating two techniques into a single framework. First, to reduce the transmission volume, the proposed algorithm employs a unified compression scheme inspired by~\cite{Pushi,TIT} that allows both relative and absolute compression errors.  Second, to further decrease the communication frequency, this algorithm incorporates a dynamic period skipping technique, which generalizes the fixed period skipping mechanisms in~\cite{TAC22} by extending them to a dynamically adjustable time-varying communication interval mechanism.

		\item 	We establish comprehensive convergence guarantees for the proposed algorithm. For the general case of non-contractive operators, Theorem~\ref{Theo1} demonstrates convergence to a neighborhood of the fixed point at a rate of $\mathcal{O}({\ln T}/{\sqrt{T}} )$. To the best of our knowledge, this result, derived via a storage function construction, provides among the first theoretical unification with distributed non-convex optimization algorithms~\citep{Jueyou,AdaOnline,KaiHigh,BiaseTSP}. Under the contractive operator regime, Theorem~\ref{Theo2} proves a sharper convergence rate of $\mathcal{O}({\ln T}/{T})$.  A significant theoretical advancement over the closely related work~\citep{BiaseTSP} lies in our generalization to a much broader operator framework.  While the theoretical guarantee in~\cite{BiaseTSP} restricts the state-dependent bias by requiring it to be strictly bounded by the inverse of network's heterogeneity, our derived theoretical framework fully decouples these two parameters, guaranteeing convergence for any state-dependent bias $P<1$ independent of the network's heterogeneity level.
	\end{enumerate}
	The rest of this paper is organized as follows. In Section~\ref{s2}, standing preliminary and specific problem statements are provided.  Section~\ref{SecRul} presents the communication-efficient distributed algorithm and analyzes the convergence properties. In Section~\ref{Sec: Simulation}, some numerical experiments are given to verify the theoretical findings. Finally, Section~\ref{ScCon} presents the conclusions and discusses the direction of future research.
	
		\vspace{-3mm}
	
	\textsl{Notations:}
	$\mathbb{R}^n$ and $\mathbb{R}^{n\times n}$ represent the set of n-dimensional vectors, and $n\times n$ real matrices, respectively.  $\mathbf{I}\in \mathbb{R}^{n\times n}$ denotes the identity matrix.  $Id$ represents the identity operator in $\mathbb{R}^n$. $[n]$ denotes the set $\{1,\ldots,n\}$ for any given positive integer $n$. $\mathbf{1}_n$ denotes the $n$-dimensional vector with elements being all ones. $e_i^n$ denotes the $n$-dimensional base vector where the $i$-th element is one and all other elements are zero. For any vector $x\in \mathbb{R}^n$, $\lfloor x \rfloor$ represents the element-wise floor function. Given any vectors $x,y\in \mathbb{R}^n$, $\left <x ,y \right>$ is the inner product of $x$ and $y$. And $\| x\|$ is the standard Euclidean norm of $x$. For a matrix $A \in \mathbb{R}^{n\times n}$, $\|A\|_{\text{F}}$ and $\|A\|_2$ denote its Frobenius norm and spectral norm, respectively. Given an operator $\mathcal{T}:\mathbb{R}^n\rightarrow \mathbb{R}^n$, $\text{Fix}(\mathcal{T})=\{x\in \mathbb{R}^n|\mathcal{T}(x)=x\}$ represents the set of fixed points of operator $\mathcal{T}$. $\int_{\gamma_{yx}} \mathcal{T}(u)du$ denotes the path integral of $\mathcal{T}$ along the flat curve $\gamma_{yx}$ which starts from $y$ and ends at $x$. $\otimes$ is the Kronecker product.  \vspace{-3mm}

	\section{Preliminaries and Problem Formulation}\label{s2} \vspace{-3mm}
	In this section, we first introduce some preliminary knowledge, and then present the formulation of the  distributed sum-separable operators along with the related standard assumptions, and finally state the communication-efficient schemes. 				\vspace{-3mm}
	 
	\subsection{Graph Theory} \vspace{-2mm}
	Let $\mathscr{G}=\{[N],\mathcal{E}\}$ represent the weight graph, where $[N]$ denotes the set of agents and $\mathcal{E}\subseteq [N]\times [N]$ represents the edges set. The path is defined by a sequence of edges $(i_1,i_2),(i_2,i_3),\ldots,(i_{k-1},i_{k})$. Then, $\mathscr{G}$ is said to be strongly connected if there exists a  path between any two agents. Define $\mathcal{N}_{i}^{\text{out}}=\{j\in[N]|(i,j)\in \mathcal{E} \}$ and $\mathcal{N}_{i}^{\text{in}}=\{j \in[N] \mid(j, i) \in  \mathcal{E} \}$ as the sets of out- and in-neighbors of agent $i$, respectively. Define $\mathbf{W}=[w_{ij}]\in \mathbb{R}^{N\times N}$ as the mixing matrix associated to  graph $\mathscr{G}$, where the element $w_{ij}>0$ if $(i,j)\in \mathcal{E}$ and $w_{ij}=0$ otherwise. The mixing matrix $\mathbf{W}$ is said to be doubly stochastic if $\mathbf{W}\mathbf{1}_N=\mathbf{W}^\top\mathbf{1}_N=\mathbf{1}_N$. To facilitate the subsequent analysis, let $\alpha=\|\mathbf{I}-\mathbf{W}\|_2$. Define $\lambda_1(\mathbf{W})\ge \lambda_2(\mathbf{W})\ge \cdots \ge \lambda_N(\mathbf{W})$ as the eigenvalues of $\mathbf{W}$. Let the spectral gap be $\kappa= 1-\max \{|\lambda_2(\mathbf{W})|,|\lambda_N(\mathbf{W})|\}$. Based on~\cite{CaoAuto,Jueyou}, it holds that $0 <\kappa \le 1$ and $1-\kappa =\|\mathbf{W}-\frac{1}{N}\mathbf{1}_N\mathbf{1}_N^\top\|$.

	\vspace{-3mm}
	\subsection{Problem Formulation}\label{s2.B} \vspace{-3mm}
	Consider a distributed sum-separable global operator over the network of $N$ agents as  \vspace{-3mm}
	\begin{align}\label{F0}	
		\mathcal{T}(x)=\frac{1}{N}\sum_{i=1}^{N} \mathcal{T}_i(x),
	\end{align}
	where $\mathcal{T}:\mathbb{R}^n\rightarrow \mathbb{R}^n$ is the global operator and  $\mathcal{T}_i:\mathbb{R}^n\rightarrow \mathbb{R}^n$ denotes the local operator privately accessible only by agent $i\in [N]$. The communication network is modeled by a graph $\mathscr{G}=\{[N],\mathcal{E}\}$. Our main aim is to develop a framework of distributed algorithms for seeking the fixed points of operator $\mathcal{T}$, i.e.,
	\begin{align}\label{F00}
		\text{Find}\quad x^*\in \text{Fix}(\mathcal{T}).
	\end{align}
	\vspace{-7mm}

	Each local operator $\mathcal{T}_i$ is allowed to have the stochastic representation as 
	\begin{align}\label{F1}
		\mathcal{T}_i(x)=\mathbb{E}_{\xi_{i}}[\hat{\mathcal{T}}_i(x,\xi_i)]=\int_{\Omega_i}\hat{\mathcal{T}}_i(x,\xi_{i})d\mathbb{P}_i,
	\end{align}
	where $\hat{\mathcal{T}}_i(x,\cdot)$ is Lebesgue integrable for $x\in \mathbb{R}^n$ and $\xi_i: \Omega_i \rightarrow \mathbb{R}^n$ is a random vector defined on a probability space $(\Omega_i,\mathcal{F}_i,\mathbb{P}_i)$. In our distributed setting, distributions $\{\mathbb{P}_i\}_{i\in [N]}$ do not have to be identical, and are permitted to be heterogeneous. Then we have 
	\begin{align*}
		\mathcal{T}(x)=\mathbb{E}_{\xi}[\hat{\mathcal{T}}(x,\xi)],
	\end{align*}
	where $\xi=(\xi_1,\ldots,\xi_N)$ and $\hat{\mathcal{T}}=\frac{1}{N}\sum_{i=1}^{N}\hat{\mathcal{T}}_i$. Similar to \cite{XiuAuto,DOT}, we assume that the set of fixed points Fix($\mathcal{T}$) is nonempty. The operator $\mathcal{T}$ is said to be $L$-Lipschitz continuous if it holds $\|\mathcal{T}(x)-\mathcal{T}(y)\| \leq L\|x-y\|, \forall x, y \in \mathbb{R}^n$ with a constant $L > 0$. $\mathcal{T}$ is termed a contractive operator if $L<1$, and it is termed a non-contractive operator if $L \geqslant 1$ which includes the non-expansive case with $L=1$ and the expansive case with $L>1$. Besides, in the non-contractive case, the operator may have multiple fixed points, and when the operator is contractive, the fixed point is unique by Banach's theorem~\citep{BB1952}. Similar to~\cite{PointTac}, we focus on the class of operators where $\mathcal{T}_{i}$ is a conservative vector field. Our main concern is to find a fixed point $x^*\in \text{Fix}(\mathcal{T})$  in a distributed manner. 				\vspace{-3mm}
	
	\begin{remark}
		The distributed sum-separable structure in (\ref{F0}) is first proposed in~\cite{XiuAuto} and does not necessitate the existence of the common fixed point for all local operators $\mathcal{T}_i$, which distinguishes it from the distributed common  fixed point finding problem in~\cite{for2,for3,XiuTAC}. In fact, when the local operators in~(\ref{F0}) are restricted to be non-expansive as in~\cite{for2,for3,XiuTAC}, the existence of a common fixed point guarantees that $\text{Fix}(\mathcal{T})=\text{Fix}(\frac{1}{N}\sum_{i=1}^{N}\mathcal{T}_i)=\bigcap_{i=1}^N{\text{Fix}(\mathcal{T}_i)}$ according to~\cite[Proposition 4.47]{book1}. Under such specific conditions, the common fixed point problem in~\cite{for2,for3,XiuTAC} essentially degenerates into a special case of problem~(\ref{F00}). 
		It is noteworthy that requiring the existence of common fixed points is equivalent to demanding all local cost functions to reach optimum at the same point in distributed optimization~\citep{TAC22,Pushi,XinleiTAC,FanTcns,Jueyou}, which is impractical in many applications. As a more general framework, problem (\ref{F00}) can characterize a multitude of interesting problems including distributed algebraic equations and distributed optimization. For instance, similar to~\cite{XiuAuto}, by introducing $\mathcal{T}_i=Id-\tau \nabla \varUpsilon_i$ with some $\tau >0$, problem (\ref{F00}) reduces to the distributed optimization problem in~\cite{Jueyou,AdaOnline}, where $\varUpsilon_i$ is the local cost function of agent $i$.
	\end{remark}

			\vspace{-2mm}
\subsection{Standard Assumptions}		\vspace{-2mm}
	We hereby impose some standing assumptions for the aforementioned distributed fixed point seeking problem.		\vspace{-4mm}
	\begin{assumption}\label{ass5.Un}
		The underlying communication graph $\mathscr{G}$ is strongly connected and its mixing matrix $\mathbf{W}$ is doubly stochastic.
	\end{assumption}
	\vspace{-3mm}
	\begin{assumption}\label{ass1.LL}
		The operator $\mathcal{T}_i$ is $L$-Lipschitz continuous with constant $L>0$, that is, for any $x,y \in \mathbb{R}^n$, $\rVert \mathcal{T}_i(x)-\mathcal{T}_i(y) \lVert\le L\rVert x-y \lVert $.
	\end{assumption}
	In the subsequent section, we focus on both the general non-contractive case involving expansive operators and the contractive case.
	\begin{remark}
		Note that Assumption \ref{ass1.LL} is rather relaxed, requiring only operators to be Lipschitz continuous. It naturally includes the special case of contractive ($L<1$) and non-expansive ($L=1$) operators in~\cite{locTac,XiuAuto,DOT,FanTcns}. Although the non-expansive condition widely adopted in the literature is sufficient to model fairly general convex optimization problems, the case where $L>1$, i.e., the expansive operator contained in Assumption~\ref{ass1.LL}, remains of great significance. Indeed, for any $x,y \in \mathbb{R}^n$ and some $m$-smooth convex function $\varUpsilon_i:\mathbb{R}^n \rightarrow \mathbb{R}$, set operator $\mathcal{T}_i= Id-\tau \nabla \varUpsilon_i$ with positive constant $\tau$,  we have 
		\begin{align}\label{TTxy}
			&\lVert \mathcal{T}_i\left( x \right) -\mathcal{T}_i\left( y \right) \rVert ^2\nonumber\\
			&= \lVert x-y \rVert ^2-2\tau \left < x-y,\nabla \varUpsilon_i \left( x \right) -\nabla \varUpsilon_i \left( y \right) \right >\nonumber\\
			&\quad +\tau ^2\lVert \nabla \varUpsilon_i \left( x \right) -\nabla \varUpsilon_i \left( y \right) \rVert ^2\nonumber\\
			&\le \lVert x-y \rVert ^2+(\tau ^2-\frac{2\tau}{m})\lVert \nabla \varUpsilon_i \left( x \right) -\nabla \varUpsilon_i \left( y \right) \rVert ^2\nonumber\\
			&\le \max \left\{1, 1-2\tau m+\tau ^2m^2  \right\} \lVert x-y \rVert^2,
		\end{align}
		where the second inequality leverages the $\frac{1}{m}$ co-coercive property of $\nabla \varUpsilon_i$. Hence, it follows that $\mathcal{T}_i$ is non-expansive for constant $\tau \in (0,\frac{2}{m}]$. Further, if $\varUpsilon_i$ is $\mu$-strongly convex, then for the corresponding operator $\mathcal{T}_i$, based on the first inequality of  (\ref{TTxy}), we have 
		\begin{align}
			\lVert \mathcal{T}_i\left( x \right) -\mathcal{T}_i\left( y \right) \rVert ^2\le \left( 1-2\tau \mu+\tau ^2m^2  \right)\lVert x-y \rVert ^2.
		\end{align}
		Consequently, when constant $\tau \in (0, \frac{2\mu}{m^2})$, $\mathcal{T}_i$ is contractive.
		Nevertheless, when the cost function $\varUpsilon_i$ is non-convex, even if it satisfies the Polyak--{\L}ojasiewicz condition, the corresponding operator $\mathcal{T}_i$ fails to satisfy the non-expansive condition as shown in~\cite{XiuAuto,DOT,FanTcns}, implying the inability to solve non-convex optimization problems leveraging the fixed point seeking algorithms in~\cite{XiuAuto,DOT,FanTcns}.  For instance, consider a simple non-convex function $\varUpsilon_i(x)=x_1^2-x_2^2$ with $x=(x_1,x_2)\in \mathbb{R}^2$, its corresponding operator $\mathcal{T}_i(x)=((1-2\tau )x_1, (1+2\tau)x_2)$ has the Lipschitz constant $L=\max \{|1-2\tau|,|1+2\tau |\}>1$ for any $\tau >0$. It should be emphasized that the investigation of expansive operators is relatively scarce in both distributed and centralized situations.
	\end{remark}
		\vspace{-3mm}
	\begin{assumption}\label{ass.heteroge}
		For local operators $\{ \mathcal{T}_i \}_{i\in [N]}$, there exists a constant $\zeta \geq 0$ such that 
		\begin{align}
			\frac{1}{N}\sum_{i=1}^N \|\mathcal{T}_i(x) - \mathcal{T}(x)\|^2 \leq \zeta^2, \quad \forall x \in \mathbb{R}^n. \label{eq:operator_heterogeneity}
		\end{align}
	\end{assumption}
		\vspace{-5mm}
	It is well known that computing multidimensional integral (\ref{F1}) with possibly unknown distribution $\mathbb{P}_i$ exactly or with high precision is intractable. Besides, while the operator $\mathcal{T}_i$ is the exact expectation of $\hat{\mathcal{T}}_i$, in practice, agents may only have access to a biased oracle, denoted by $\widetilde{\mathcal{T}}_i(x, \xi_i)$. To this end, similar to the stochastic approximate technique adopted in~\cite{PointTac,LanGuan,ICMLBias}, we make the following assumption to obtain a general form of biased operator estimators.
		\vspace{-3mm}
	\begin{assumption}\label{ass2.YY}
		For each operator $\mathcal{T}_i$, there exists a stochastic oracle that provides a sampled operator $\widetilde{\mathcal{T}}_i(x,\xi_i)$ given $x,\xi_i$ as the noisy evaluation of $\mathcal{T}_i$. And it satisfies that 
		\begin{align*}
			\left\|\mathbb{E}_{\xi_i}[{\widetilde{\mathcal{T}}}_i(x, \xi_i)]-\mathcal{T}_i(x)\right\|^2 &\leq \beta^2+P\|\mathcal{T}_i(x)-x\|^2, \\
			\mathbb{E}_{\xi_i}\|\widetilde{\mathcal{T}}_i(x, \xi_i)-\mathcal{T}_i(x)\|^2 &\leq \sigma^2+M\|\mathcal{T}_i(x)-x\|^2,
		\end{align*}
		for some constants $P<1$ and $M, \beta, \sigma \geq 0$.
	\end{assumption}
		\vspace{-3mm}
	\begin{remark}
		It is easy to see that the growth conditions in Assumption~\ref{ass2.YY} are much less restrictive than the unbiased and bounded additive variance assumptions typically adopted in stochastic optimization~\citep{KaiHigh,TAC22,YuanAuto,Jueyou,Yongqiang2023,XinleiAuto}, which merely correspond to a restrictive special case of our framework with $P=0$, $\beta=0$ meaning unbiasedness, and $M=0$ meaning only additive variance.  Here, $\beta$ and $\sigma$ represent the constant bias and variance, while $P$ and $M$ flexibly  control the state-dependent growth of bias and variance, respectively.	
		Similar to~\cite{Lei13,PointTac}, the condition $P<1$ is necessary to ensure that the state-dependent bias diminishes sufficiently fast relative to the progress of the iterates, preventing divergence.
	\end{remark}
		\vspace{-3mm}
	\begin{remark}
		Abundant examples can be viewed as special cases of the biased stochastic framework given in 	Assumption~\ref{ass2.YY}. An important example of the biased oracle is the zeroth-order gradient estimator obtained through Gaussian smoothing~\citep{freeproject,Jueyou}. It is valuable for black-box model optimization without direct gradient access or high-dimensional models with costly gradient computation due to large parameter scales.
		Given a $m$-smooth function $\varUpsilon:\mathbb{R}^n\rightarrow \mathbb{R}$, the 2-point zeroth-order estimator~\citep{Jueyou,XinleiAuto} of $\nabla \varUpsilon(x)$ is defined as 				\vspace{-3mm}
		\begin{align}\label{ZePon}
			G_{\varUpsilon}(x,u)=\frac{\varUpsilon(x+zu)-\varUpsilon(x)}{z}u,
		\end{align}
		where $z>0$ is the smoothing radius parameter, $u$ is a random vector sampled from the distribution $\mathcal{U}(\mathbb{S}^{n-1})$, and unit sphere $\mathbb{S}^{n-1}=\{x\in \mathbb{R}^n|\|x\|=1\}$. According to~\cite[Lemma 3 and Theorem 4]{freeproject}, it can be concluded that (\ref{ZePon}) is a biased estimate of $\nabla \varUpsilon(x)$ and Assumption \ref{ass2.YY} holds with $\beta^2=\tau ^2{z^2}m^2(n+3)^2/{4}$, $P=0$, $\sigma^2=3\tau ^2z^2m^2(n+4)^3$ and $M=4(n+4)$, where $\mathcal{T}(x)=x-\tau \nabla \varUpsilon (x)$, $\widetilde{\mathcal{T}}(x,\xi)=x-\tau G_{\varUpsilon}(x,u)$ and $\tau$ is the stepsize parameter. Moreover, as a direct application of operator theory, stochastic optimization often encounters biased gradient estimates in various settings. Examples include stochastic algorithms with state-dependent sampling distributions~\citep{Sto11}, e.g., sampling from Markov decision processes~\citep{Sto22,Sto33} and random reshuffling sampling procedures~\citep{Sto44,Sto55}. 
	\end{remark}
		\vspace{-3mm}
	\begin{assumption}\label{ass4.YJ}
		There exists a positive constant $\mathcal{D}$ such that $\mathbb{E}_{\xi_i} \left[\left \|\widetilde{\mathcal{T}}_i(x,\xi_{i})-x\right \|^2 \right]\le \mathcal{D}^2$ for any $x\in \mathbb{R}^n$.
	\end{assumption}
		\vspace{-3mm}
	This assumption essentially bounds the expected squared distance between the stochastic operator's evaluation and the current state. 	In the context of optimization where $\mathcal{T}_i(x)=x-\tau \nabla \varUpsilon_i(x)$, it is equivalent to the bounded stochastic gradient assumption adopted in~\cite{XiuAuto,Zeng11}.

			\vspace{-2mm}
	\subsection{Communication-Efficient Scheme}\label{s2.C}		\vspace{-2mm}
	To alleviate the substantial communication overhead inherent in distributed algorithms, we pursue efficiency from two complementary perspectives. We employ communication compression to reduce the volume of transmitted data and introduce a dynamic period skipping mechanism to decrease the transmission frequency.
	
	To implement the data volume reduction, we consider a general class of unified compressors accommodating both relative and absolute errors. This scheme satisfies the following condition.
			\vspace{-3mm}
	\begin{assumption}\label{ass3.YA}
		The compressor $\mathcal{C}:\mathbb{R}^n \rightarrow \mathbb{R}^n$
		satisfies		\vspace{-3mm}
		\begin{align}\label{222}
			\mathbb{E}_{\mathcal{C}}\left[\left\|\frac{\mathcal{C}(x)}{r}-x\right\|^2\right] \leq(1-\varphi )\|x\|^2+\delta^2, \forall x \in \mathbb{R}^n,
		\end{align}
		for some constants $r>0,\varphi\in (0,1]$ and $\delta\ge 0$. Here, $\mathbb{E}_{\mathcal{C}}$ denotes the expectation over $\mathcal{C}$. 
	\end{assumption}

	The unified compression scheme introduced in (\ref{222}) accommodates simultaneously relative and absolute compression errors, enhancing both practical utility and theoretical generality compared to prior works~\citep{Jueyou,CaoAuto,TAC22,XinleiTAC,ZhangjiaqiTAC} limited to single-type compression errors. This framework supports broader compression techniques. For instance, it reduces to the standard relative error compressor  in~\cite{Jueyou,CaoAuto,TAC22} when $\delta=0$. If $\varphi=1,$ i.e., there is no relative error, then it degenerates to the compressor with absolute errors in~\cite{XinleiTAC,ZhangjiaqiTAC}. Crucially, Assumption \ref{ass3.YA} not only eliminates conventional requirements for unbiasedness in relative error compression~\cite{Houcsl,Yongqiang2023}
	but also removes the necessity for contractivity in compressors~\citep{Jueyou,CaoAuto,TAC22,ZhangjiaqiTAC}. To see this, consider the case of $\delta=0$ as in ~\cite{Jueyou,CaoAuto,TAC22,ZhangjiaqiTAC}, for $ x\in \mathbb{R}^n$, the following can be derived from (\ref{222})
			\vspace{-6mm}
	\begin{align}\label{SJ}
		&\mathbb{E}_{\mathcal{C}}\left[\|\mathcal{C}(x)-x\|^2\right]=\mathbb{E}_{\mathcal{C}}\left[\left \|r\left(\frac{\mathcal{C}(x)}{r}-x\right )  +(r-1)x\right \|^2\right]\nonumber\\
		&\le 2r^2\mathbb{E}_{\mathcal{C}}\left[\left \|\frac{\mathcal{C}(x)}{r}-x\right \|^2\right]+2(1-r)^2\|x\|^2\nonumber\\
		&\le  \bar{r}\|x\|^2,
	\end{align}
	where $\bar{r}=2r^2(1-\varphi )+2(1-r)^2$. Evidently, compared to~\cite{Jueyou,CaoAuto,TAC22,ZhangjiaqiTAC}, which require $\mathcal{C}$ to be contractive~(i.e., $\bar{r}<1$),~(\ref{SJ}) imposes no such restrictions. 
	
	To further curtail communication complexity, we additionally implement a dynamic period skipping mechanism that divides the total iterations into variable-sized blocks. The algorithmic details of this frequency reduction strategy will be systematically presented in Section~\ref{SecRul}.

	\vspace{-4mm}
	\section{Main result}\label{SecRul} 		\vspace{-3mm}
	In this section, we first design the communication-efficient distributed algorithm for finding fixed points and then state the corresponding convergence analysis results.
		\vspace{-3mm}
	\subsection{Algorithm Design}
			\vspace{-3mm}
	Before giving a concrete description of the algorithm, it is necessary to emphasize that the majority of the works related to fixed points seeking concern how to deliver theoretical insights into optimization algorithms from the perspective of the operator framework. Whereas in this paper confronted with more general biased stochastic operators even permitted to be expansive, it is from a complementary viewpoint that we leverage advanced stochastic optimization theory to investigate fixed point iteration algorithms. 
	
	We first review the inexact Krasnosel'ski\u{\i}--Mann (KM)  iteration~\citep{PointTac} of the form
	\begin{align}\label{KKM}
		x^{t+1}=\left(1-\eta_t\right) x^t+\eta_t \widetilde{\mathcal{T}}\left(x^t, \xi^t\right),
	\end{align}
	where $\left\{\eta_t\right\}$ is a sequence of relaxation parameters with $\eta_t \in(0,1)$. This is a classical centralized algorithm capable of seeking the fixed points in the presence of random operator noise, which can be viewed as the generalization of the method in~\cite{KM1955} for non-expansive operators.

	\begin{algorithm}[H]
		\caption{Communication-Efficient Distributed Fixed Point Iteration} \label{alg.1}
		\begin{algorithmic}[1]
		   \State\textbf{Require:}{Initialize $x_i^0=\hat{x}_i^0=0, i\in [N]$; Choose suitable step sizes $\eta_t$, $\gamma$, $\psi$ and scaling parameter $s_t$;}
			\For{$t=0,1,\cdots,T$}
			\For{$i=1,2,\cdots,N$}
			\State Sample $\xi_i^{t}$ and compute $\widetilde{\mathcal{T}}_i(x_i^{t}, \xi_i^{t})$.
			\State $z_i^{t+1} = (1- \eta_t)x_i^{t} + \eta_t \widetilde{\mathcal{T}}_i(x_i^{t},\xi_i^{t})$.
			\If{$t+1 \in \mathcal{I}_T$}
			\State $x_i^{t+1} = z_i^{t+1} + \gamma \sum\limits_{j=1}^N w_{ij} (\hat{x}_j^{t} - \hat{x}_i^{t} )$.
			\State Compute $c_i^{t} := \mathcal{C}((x_i^{t+1} - \hat{x}_i^{t} )/s_t)$.
			\State Send $c_{i}^t$ to all neighbors $j\in\mathcal{N}_i^{\text{out}}$.
			\State Receive $c_{j}^t$ from all neighbors $j\in\mathcal{N}_i^{\text{in}}$.
			\State Update the auxiliary vector:
			\State $\hat{x}_j^{t+1}=\hat{x}_j^{t} +\psi s_tc_j^{t}, \quad j\in \mathcal{N}_i^{\text{in}}\cup\{i\}$.
			\Else
			\State $x_i^{t+1} = z_i^{t+1}$, $\hat{x}_j^{t+1} = \hat{x}_j^{t}$, $j\in\mathcal{N}_i^{\text{in}}\cup\{i\}$.
			\EndIf
			\EndFor
			\EndFor
			\State\textbf{Output:} $\{x_i^T\}$
		\end{algorithmic}
	\end{algorithm}

		\vspace{-3mm}
	
	Now, invoking (\ref{KKM}), we are in the position of presenting the communication-efficient distributed 	algorithm with compressed and dynamic period skipping communication. The pseudo-code is presented in Algorithm~\ref{alg.1} from the perspective of agent $i$. Therein, each agent maintains and updates local state variable $x_i^t$ via line 5 and line 7 for seeking the fixed point $x^*$. Meanwhile, an estimate $\hat{x}_i^t$ of $x_i^t$ is also maintained at each neighbor $j \in \mathcal{N}_i$ and at $i$ itself. At each time $t$, following the local KM iteration in line 5, if $t+1 \in \mathcal{I}_{T}$, agent $i$ executes the consensus step in line 7, and sends the compressed difference between its local state and its estimate to all out-neighbors via line 8 and line 9.  Then based on the compressed messages received from its neighbors, agent $i$ updates $\hat{x}_j^t$ via line 12. If $t+1 \notin \mathcal{I}_T$, agents will not communicate with neighbors.
	
		\vspace{-3mm}
	
	In the proposed distributed algorithm, the agents are not expected to transmit information with their neighbors at every iteration. Rather, communication is skipped during certain time intervals. In order to provide a comprehensive description of the dynamic period skipping communication  mechanism, let $\mathcal{I}_{T} \subseteq [T]$ signify the set of indices at which agents engage in interactions. Define $\mathcal{I}_T=\{\mathcal{I}_{(1)},\mathcal{I}_{(2)},\ldots, \mathcal{I}_{(k)},\ldots\}$ with $\mathcal{I}_{(1)}=1$ and gap($\mathcal{I}_T$):=${\max}_{l \in [k-1]}\{\mathcal{I}_{(l+1)}-\mathcal{I}_{(l)}\} \le \mathcal{H}$ which is used to assess the maximum number of local iterations between two communication moments, and $\mathcal{H}$ can represent the maximum communication interval. The communication intervals are non-fixed, meaning agents employ dynamic period communication scheduling. If $\text{gap}(\mathcal{I}_T)=1$, then it reduces to a standard distributed setup with per-iteration communication. Increasing $\mathcal{H}$ leads to less frequent communication. Notably, the communication skipping mechanism exhibits conceptual parallels with local update steps in~\cite{LocalJ,LocalSun}, where the communications skipped and omitted between $\mathcal{I}_{(t)}$ and $\mathcal{I}_{(t+1)}$ in Algorithm~\ref{alg.1} essentially correspond to performing the cumulative local update $\sum_{t'=\mathcal{I}_{(t)}}^{\mathcal{I}_{(t+1)}-1}{x_i^{t'}-\widetilde{\mathcal{T}}_i^{t'}\left( x_{i}^{t'},\xi _{i}^{t'} \right)}$. Here, the interval $\mathcal{I}_{(t+1)}-\mathcal{I}_{(t)}$ serves an analogous role to the number of additional local updates in these existing methods. While both approaches aim to reduce communication frequency, their implementations differ fundamentally. Specifically, unlike the fixed interval local updates required by the aforementioned studies, our communication skipping mechanism enables dynamic interval adaptation based on network conditions or model variations (e.g., shorter intervals during early iterations for accelerated convergence, longer intervals later to save communication), enhancing flexibility in balancing communication and computation costs. Meanwhile, Algorithm~\ref{alg.1} further integrates compression to reduce data volume per transmission, thereby further improving overall communication efficiency.

		\vspace{-3mm}

	It is worthwhile to mention that compressor $\mathcal{C}$ satisfying Assumption~\ref{ass3.YA} is characterized by a combination of both relative and absolute compression errors. Correspondingly, in an attempt to avoid the possible divergence associated with the absolute error term $\delta$, we employ the dynamic scaling compression technique~\citep{Quant,Pushi,XinleiTAC} in Algorithm~\ref{alg.1}.
	The cumulative effects of arbitrarily large absolute compression errors are effectively mitigated by utilizing the exponentially decaying dynamic parameter $s_t$. To see this, define $y_i^t=x_i^{t+1}-\hat{x}_i^t$ and we derive that 
			\vspace{-3mm}
	\begin{align*}
		\mathbb{E}_{\mathcal{C}}\left[\left\|\frac{\mathcal{Q}(y_i^t)}{r}-y_i^t\right \|^2 \right]=&s_t^2\mathbb{E}_{\mathcal{C}}\left[\left\| \frac{\mathcal{C}(y_i^t/s_t)}{r} -\frac{y_i^t}{s_t}\right \|^2\right]\nonumber\\
		\le&s_t^2\left ((1-\varphi)\left\| y_i^t/s_t\right\|^2+\delta^2\right)\nonumber\\
		=&(1-\varphi)\left\| y_i^t\right\|^2+s_t^2\delta^2,
	\end{align*}
	where $\mathcal{Q}(y_i^t)=s_t\mathcal{C}(y_i^t/s_t)$ can be viewed as the scaled version of compressor $\mathcal{C}$. Therefore, by decaying the parameter $s_t$, the absolute error can be controlled.
	
		\vspace{-2mm}
	\subsection{Surrogate function}\label{SecRul.B}
			\vspace{-2mm}
	To facilitate the subsequent convergence analysis,  we establish the connection between stochastic operator and stochastic optimization theory by introducing the surrogate function.
	
		\vspace{-3mm}
	
	In view of the Lebesgue integrability of operator $\mathcal{T}$, define the surrogate function of $\mathcal{T}$ as 
	\begin{align}\label{Zf}
		\mathcal{G}(x):=\int_{\gamma_{yx}} u-\mathcal{T}(u)du,
	\end{align}
	where $\gamma_{yx}$ is a continuous flat curve originating at some point $y$ and terminating at point $x$. Combined with Assumption~\ref{ass1.LL}, it is easy to verify that $\nabla \mathcal{G}(x)=x-\mathcal{T}(x)$ is $(1+L)$-Lipschitz continuous. The bridge to the optimization theory is built by the surrogate function $\mathcal{G}$, which is essential in the subsequent analysis. To this end, define an optimization problem $\underset{x\in \mathbb{R}^n}{\min} \mathcal{G}(x)$.
	Appealing to $\nabla \mathcal{G}(x)$, it can be concluded that the first-order optimality condition of (\ref{Zf}) can be recast as 
	\begin{align*}
		\left\langle\nabla \mathcal{G}\left(\tilde{x}^*\right), \tilde{x}^*-x\right\rangle \leq 0, \quad \text { for all } x \in \mathbb{R}^n,
	\end{align*}
	where $\tilde{x}^*$ is committed both to a first-order solution of (\ref{Zf}) and to a fixed point of $\mathcal{T}$. Based on~\cite[Proposition 17.10 and Theorem 27.1]{book1}, the  first-order stationary points of (\ref{Zf}) construct a subset of the fixed points of $\mathcal{T}$ and further when the operator $\mathcal{T}$ is non-expansive, the two are equivalent. 
	Evidently, given the definition of $\mathcal{T}$, the function $\mathcal{G}$ can be formalized as 
	\begin{align*}
		\mathcal{G}(x)=\frac{1}{N}\sum_{i=1}^{N}\mathcal{G}_i(x), \quad \mathcal{G}_i(x)=\mathbb{E}_{\xi_i}[\hat{\mathcal{G}}_i(x,\xi_i)],
	\end{align*}
	where $\hat{\mathcal{G}}_i(x,\xi_{i})=\int_{\gamma_{yx}}u-\hat{\mathcal{T}}_i(u,\xi_i)du$.
	
		\vspace{-3mm}
	
	As a consequence, by introducing surrogate function $\mathcal{G}_i$, it is enabled to explore the analysis of fixed point iteration algorithms for stochastic operators leveraging the biased stochastic optimization theory. Based on~\cite{PointTac}, in the following lemma, some useful properties of the surrogate function are sketched by simple derivation of the features of $\mathcal{T}_i$.

			\vspace{-3mm}
	\begin{lemma}\label{lemm.aa} (Proposition 1 in \cite{PointTac})
		The surrogate function $\mathcal{G}_i$ is $(1+L)$-smooth and thus satisfies
		\begin{enumerate}
			\item[(a)] $\mathcal{G}_i(y) \leq \mathcal{G}_i(x) + \langle\nabla \mathcal{G}_i(x), y-x\rangle + \frac{1+L}{2}\|x-y\|^2$,
		\end{enumerate}
		for all $x, y \in \mathbb{R}^n$. Additionally, if $L < 1$, then $\mathcal{G}_i$ is $(1-L)$-strongly convex and thus satisfies
		\begin{enumerate}
			\item[(b)] $\|\mathcal{T}_i(x)-x\|^2 \geq 2(1-L)[\mathcal{G}_i(x)-\mathcal{G}_i(x_i^*)]$,
			\item[(c)] $\mathcal{G}_i(x)-\mathcal{G}_i(x_i^*) \geq \frac{1-L}{2}\|x-x_i^*\|^2$,
		\end{enumerate}
		for all $x \in \mathbb{R}^n$, where $x_i^*$ is the fixed point of $\mathcal{T}_i$.
	\end{lemma} 
			\vspace{-3mm}
	\subsection{Convergence Analysis}\label{SecRul.C}		\vspace{-3mm}
	Before presenting the convergence results, we highlight the methodological role of the surrogate function~(\ref{Zf}). Specifically, tracking the trajectory of general non-contractive operators is analytically intractable. Instead, our framework leverages the expected function~(\ref{Zf}) as a generalized storage function. By rigorously invoking the connection $\nabla \mathcal{G}(x)=x-\mathcal{T}$, the surrogate function $\mathcal{G}$ acts as the fundamental theoretical bridge that makes the analysis of non-contractive operators viable. 
	
		\vspace{-3mm}
	
	We are now ready to present the main result of the general non-contractive case.
	\begin{theorem}\label{Theo1}
		Suppose that Assumptions \ref{ass5.Un}--\ref{ass3.YA} hold. Consider the sequence $\{x_i^t\}$ generated by Algorithm \ref{alg.1}, where parameter $\gamma$ satisfies
		\begin{align}\label{Gma}
			0 <& \gamma < \min \bigg\{  \varphi \left( 8\left( 1+\frac{4}{\kappa } \right) \alpha^2+16\left( 1-\frac{\varphi}{4} \right) \alpha \right) ^{-1}, \nonumber\\
			& \frac{3\kappa}{2}\left( \frac{9\kappa ^2}{16}+2\left( 1-\psi r\varphi \right) \left( 1+\frac{4}{\varphi} \right) \alpha^2 \right) ^{-1} \bigg \},
		\end{align}
		with $\psi \in (\frac{3}{4r},\frac{1}{r}]$. Let the time-varying step sizes and compression scaling parameters be chosen as $\eta_{t} = s_t = \frac{b}{\sqrt{t+a}}$ for all $t \ge 0$. The constant parameters $a$ and $b$ are selected such that $a > \frac{4\mathcal{H}}{3\zeta_1(\gamma)}$ and $0 < b \le \frac{(1-P)\sqrt{a}}{6(1+L)(1+4M)}$,
		with $\zeta_1(\gamma)$ explicitly defined in~\ifbool{submitA}{Appendix 6.1 of~\cite{Mypaper}
		}{\eqref{Zt1}}. Then, for any  $T \ge 1$, the expected network consensus error satisfies
		\begin{align}\label{equ.Theo1_consensus}
			\sum_{t=1}^T \sum_{i=1}^N \mathbb{E}\left[\|x_i^t - \bar{x}^t\|^2\right] \leq C_{1} b^2 \ln\left(1 + \frac{T}{a}\right), 
		\end{align}
		where $\bar{x}^t=\frac{1}{N}\sum_{i=1}^{N}x_i^t$ and the positive constant $C_{1}$ is given in~\ifbool{submitA}{given in Appendix 6.1 of \cite{Mypaper}.
		}{(\ref{equ.Ce_def}).}
		Furthermore, the sequence $\{x_i^t\}$ of Algorithm~\ref{alg.1} satisfies
		\begin{align}\label{eq.TT}
			&\frac{1}{T} \sum_{t=1}^T \frac{1}{N} \sum_{i=1}^N \mathbb{E} \left[ \| x_i^t - \mathcal{T}(x_i^t) \|^2 \right] \nonumber\\
			&\le \mathcal{O}\left( \frac{C_{1}\ln T}{\sqrt{T}} \right) + \mathcal{O}\left(  \frac{\beta^2+2P\zeta^2}{P+1} \right).
		\end{align}
	\end{theorem}
	\vspace{-7mm}
	
	\ifbool{submitA}{
\begin{proof}
The proof is given in Appendix 6.2 of the arXiv version \cite{Mypaper} and is omitted here due to space limitation.
\end{proof}
}{
\begin{proof}
The proof can be found in Appendix~\ref{FE}.
\end{proof}
}

		\vspace{-2mm}
	\begin{remark}\label{R333}
		According to the bound in Theorem 1, the steady-state error neighborhood $\mathcal{O}\left({(\beta^2+2P\zeta^2)}/({P+1})\right)$ is entirely independent of the communication skipping interval $\mathcal{H}$. This theoretical insight reveals that one can substantially decrease the communication frequency by employing a larger maximum skipping interval $\mathcal{H}$ without enlarging the final convergence neighborhood. On the other hand, setting $\mathcal{H}$ to be excessively large will inflate the transient constant  associated with the $\mathcal{O}\big( {C_{1}  \ln T}/{\sqrt{T}} \big)$ term. This behavior illustrates a practical trade-off in distributed system design, as prioritizing lower communication costs can sometimes lead to a slower early-stage convergence.
	\end{remark}
	
		\vspace{-3mm}
	
	From Theorem \ref{Theo1}, it establishes that Algorithm \ref{alg.1} asymptotically seeks fixed points within a neighborhood for arbitrary Lipschitz operators, recovering exact convergence when the bias parameters $\beta$ and $P$ vanish. 
	To the best of our knowledge, this is among the first convergence result for distributed fixed-point iterations that remains applicable even in the extended operator regime. Furthermore, for contractive operators, we can derive the following tighter convergence bounds.
			\vspace{-3mm}
	\begin{theorem}\label{Theo2}
		Suppose that Assumptions \ref{ass5.Un}--\ref{ass3.YA} hold, and further assume the local operators are contractive with a Lipschitz constant $L < 1$. Consider the sequence $\{x_i^t\}$ generated by Algorithm \ref{alg.1}, where parameter $\gamma$ satisfies the condition in \eqref{Gma}. Let the time-varying step sizes and compression scaling parameters be chosen as $\eta_{t} = s_t = \frac{b}{t+a}$ for all $t \ge 1$. The constant parameters $a$ and $b$ are selected such that $a\ge \max \left\{ \frac{8\mathcal{H}}{3\zeta_1(\gamma)}, \frac{12(1+L)(1+4M)b}{1-P} \right\}$,
		and $b > \frac{4}{(1-L)(1-P)}$. 
		Then, for any $t \ge 1$, the expected network consensus error satisfies
				\vspace{-3mm}
		\begin{align}\label{TTT22}
			\sum_{i=1}^N \mathbb{E}\big[\|x_i^t - \bar{x}^t\|^2\big] \leq \frac{C_2 b^2}{(t+a)^2},	
		\end{align}
		where the positive constant $C_{2}$ is given in~\ifbool{submitA}{Appendix 6.1 of~\cite{Mypaper}.
		}{(\ref{equ.Ce2_def}).} 
		Furthermore, the sequence $\{x_i^t\}$ of Algorithm~\ref{alg.1}  satisfies 
				\vspace{-3mm}
		\begin{align}\label{eq:final_distance_bound}
			&\frac{1}{N} \sum_{i=1}^N \mathbb{E} \big[ \|x_i^t - x^*\|^2 \big] \nonumber\\
			&\leq \mathcal{O}\left( \frac{\ln t}{t} \right) + \mathcal{O}\left(\frac{\beta^2 +2 P\zeta^2}{(1-L)^2(1-P^2)}\right).
		\end{align}
	\end{theorem}
	\vspace{-7mm}
		\ifbool{submitA}{
	\begin{proof}
		The proof is given in Appendix~6.3 of the arXiv version~\cite{Mypaper} and is omitted here due to space limitation.
	\end{proof}}{
	\begin{proof}
The proof can be found in Appendix~6.3.
	\end{proof} }
		\vspace{-2mm}
	\begin{remark}\label{R111}
		With reference to the convergence results in Theorems~\ref{Theo1} and~\ref{Theo2}, Algorithm~\ref{alg.1} achieves a sublinear convergence rate  of $\mathcal{O}(\ln T/\sqrt{T})$ for  non-contractive operators and a faster $\mathcal{O}(\ln T/T)$ rate for contractive ones. These rates seamlessly recover the optimal  iteration complexities of distributed optimization algorithms in stochastic non-convex~\citep{KaiHigh,TAC22} and strongly convex cases~\citep{YuanAuto}, respectively. Meanwhile, from an operator theoretic perspective, our rates parallel the latest results in centralized stochastic fixed-point methods~\citep{PointTac}, while successfully extending them to distributed networks. To our best knowledge, this establishes among the first distributed fixed-point iteration framework capable of addressing expansive operators (i.e., Lipschitz constant $L>1$). It fundamentally resolves the theoretical bottleneck that the operator $Id-\nabla \varUpsilon $ of a generic non-convex objective $\varUpsilon$ naturally violates the traditional non-expansiveness ($L\le 1$) requirement utilized in prior fixed-point frameworks~\citep{XiuAuto,locTac,DOT,FanTcns}. Furthermore, distinguishing itself from the case of deterministic operators with perfect communication in~\cite{XiuAuto,locTac,DOT,FanTcns}, our framework is designed for the more practical biased stochastic regime, while comprehensively mitigating communication overheads via integrated communication compression and period skipping mechanisms.
	\end{remark}
	\vspace{-3mm}
	\begin{remark}\label{R22.1}
		A crucial theoretical advancement over the recent distributed biased stochastic optimization in~\cite{BiaseTSP} lies in the improved tolerance to state-dependent bias. Specifically, the convergence criterion in~\cite{BiaseTSP} couples the allowable state-dependent bias with network's heterogeneity parameter in a form of $P\zeta <1/15$, requiring the former to shrink inversely as the latter increases, which severely limits its applicability in highly heterogeneous scenarios. Our analysis eliminates this restrictive coupling. By simply enforcing $P<1$ which coincides with the criterion in the centralized counterpart~\citep{Pushi}, the proposed framework isolates the bias tolerance from the network's heterogeneity, effectively bounding the steady-state impact of correlated bias regardless of the heterogeneity parameters $\zeta$, and  guarantees exact convergence to the fixed point under unbiased conditions  $\beta=0$ and $P=0$.
	\end{remark}
	
		\vspace{-3mm}
	\section{Numerical Examples}\label{Sec: Simulation}
			\vspace{-3mm}
	In this section, we apply the fixed-point iteration in optimization problems, validating prior theoretical results through numerical simulations. 
	
	To save communication, we consider three kinds of compressor $\mathcal{C}$ discussed in Section~\ref{s2.C} for $x\in \mathbb{R}^n$ as follows.
	\begin{itemize}
		\item $l$-bits $\infty$-quantizer		\vspace{-3mm}
		\begin{align*}
			\mathcal{C}_1(x)=\frac{\|x\|_\infty}{2^{l-1}}\text{sign}(x)\circ
			\bigg\lfloor\frac{2^{l-1}|{x}|}{\|x\|_\infty}+\varpi\bigg\rfloor,
		\end{align*}
		where $\text{sign}(x)$ is the sign function, $|{x}|$ represents the element-wise absolute value of $x$, $\circ$ is the Hadamard product, and $\varpi$ is a random perturbation vector uniformly sampled from $[0,1]^n$. As discussed in~\cite{XinleiTAC,ZhangjiaqiTAC}, $\mathcal{C}_1$ satisfies Assumption \ref{ass3.YA} with $r=1+n/4^{l}$, $\varphi=1/(1+n/4^{l})$ and $\delta^2=0$. If scalars can be transmitted with sufficient precision using $b$ bits, then transmitting $\mathcal{C}_1(x)$ requires $(l+1)n+b$ bits. In the experiments, we uniformly set $b=64$ and $l=2$.

		\item 	Standard uniform quantizer\\  		\vspace{-3mm}
		\begin{align*}
			\mathcal{C}_2(x)=\Delta\left\lfloor\frac{x}{\Delta}+\frac{\mathbf{1}_n}{2}\right\rfloor,
		\end{align*}
		where the quantization stepsize $\Delta$ is a positive integer. This compressor satisfies Assumption \ref{ass3.YA} with $r=1$, $\varphi=1$ and $\delta^2={n\Delta^2}/{4}$, and it has been adopted in several works, including~\cite{XinleiTAC,Ge2023}. When transmitting $\mathcal{C}_2(x)$, the total number of bits required is $nq$, assuming each integer is represented with $q$ bits. In the experiments, we set $\Delta=1$ and $q=8$.

		\item  The composition of sparsification and uniform quantizer $\mathcal{C}_3(x)=\mathcal{Q}_2(\mathcal{Q}_1)$, where $\mathcal{Q}_1=(q(x_1),\ldots,q(x_n))^T$ and for $k\in [n]$, 		\vspace{-3mm}
		\begin{align*}
			q(x_k)=\begin{cases}
				\frac{x_k}{p} \quad w.p.\quad  p,\\
				0    \quad w.p.\quad  1-p,
			\end{cases}
		\end{align*}
		where $p\in (0,1)$, $\mathcal{Q}_2(x)=\Delta\cdot \text{round}(\frac{x}{\Delta}+v)$ with $v$ being a random perturbation uniformly distributed in $[-{1}/{2}, {1}/{2}]^n$ and $\Delta$ is a positive integer. Then simple calculations reveal that $\mathcal{C}_3$ satisfies Assumption \ref{ass3.YA} with $r=\frac{1}{p}$, $\varphi =p$ and $\delta^2= {n p^3\Delta^2 }/{4}$, possessing both relative and absolute compression error. Only $nqp$ bits are needed to transmit $\mathcal{C}_3(x)$ if $q$ bits are allocated to transmit an integer. In the experiments, we choose $\Delta=1$ and $p=0.75$.
	\end{itemize}

		\vspace{-3mm}
	\subsection{Non-convex Case}
			\vspace{-3mm}
	We first consider a non-convex distributed optimization problem composed of $N$ agents, which is formulated as  				\vspace{-3mm}
	\begin{align}\label{OppS} 
		\underset{x\in \mathbb{R}^n}{\min} f(x)=\sum_{i=1}^{N} f_i(x),
	\end{align}
	where $N=6$ and~$f_1\left( x \right) =0.06\,x^4-0.02\,x^2$, 
	$f_2\left( x \right) =0.05\,\sin \left( x+\frac{1}{2} \right) +0.15\,\cos \left( \frac{10x}{3} \right)$, 
	$f_3\left( x \right) =0.1\,e^{-x^2}+0.1\,x^4-0.3\,x^2$, 
	$
	f_4\left( x \right) =0.14\,x^4-0.2\,x^2
	$,
	$
	f_5\left( x \right) =0.45\cos \left( x \right) +0.15\,\sin \left( \frac{10x}{3}+\frac{1}{2} \right) 
	$,
	$
	f_6\left( x \right) =0.4\,e^{-x^2}-0.3x^2
	$. Considering measurement errors and other practical factors, each agent $i$ only has access to the noisy gradient information $\nabla f_i(x(t))+\xi_i^t$ at each iteration, where $\xi_i^t$ is the measurement noise at iteration $t$. In the experiments, we let $n=30$.
	
		\vspace{-3mm}
	
	To bridge problem (\ref{OppS}) and the sum-separable model (\ref{F00}), define $\mathcal{T}_i=Id-\tau \nabla f_i$ with $\tau =0.1$ as illustrated in Section~\ref{s2.B}, enabling the proposed distributed fixed point iteration to solve stochastic optimization (\ref{OppS}). From the expression of the non-convex function $f_i$, it is straightforward to infer that $\mathcal{T}_i$ is expansive, exceeding the theoretical scope supported by existing distributed operator frameworks~\citep{XiuAuto,DOT,FanTcns,locTac}.
	
	
	\begin{figure}[!t]
		\centering 	
		\subfloat[Evolutions with compressor $\mathcal{C}_1$]{\includegraphics[width=1.5in,height=1.2in]{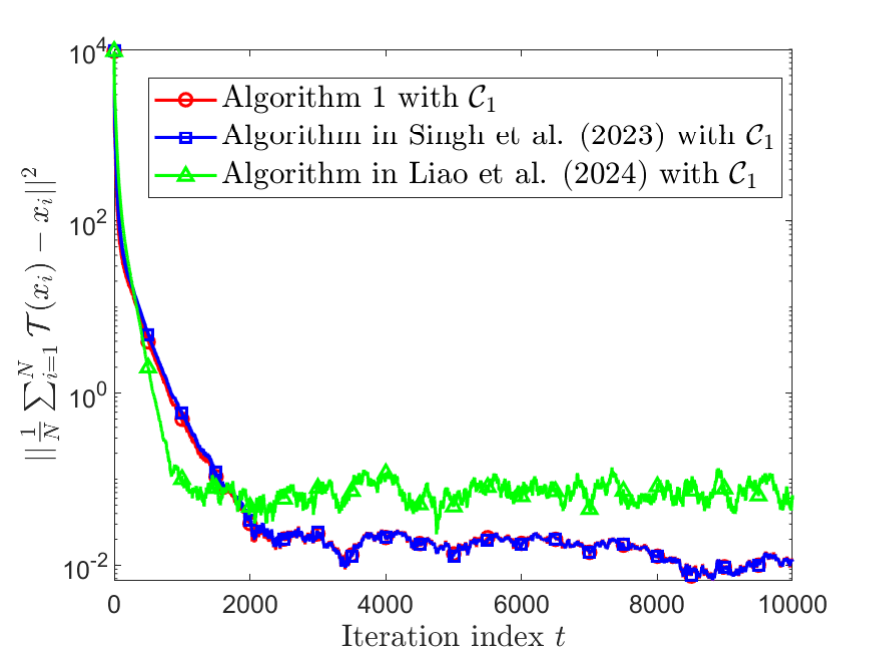}}
		\quad \quad
		\subfloat[Evolutions with compressor $\mathcal{C}_2$]{\includegraphics[width=1.5in,height=1.2in]{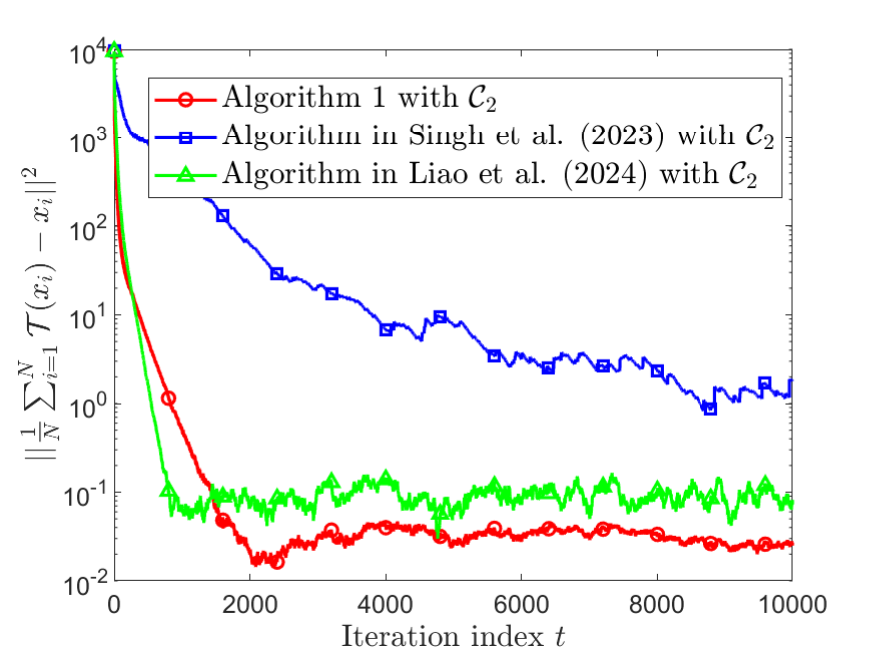}}
		\caption{Evolutions of $\| \frac{1}{N}\sum_{i=1}^{N} \mathcal{T}(x_i)-x_i \|^2$ for the three algorithms with different compressors in the non-convex case.}
		\label{FF1}
	\end{figure}
	
	\begin{figure}[!t]
		\centering 	
		\subfloat[Evolutions with respect to the number of iterations]{\includegraphics[width=1.5in,height=1.2in]{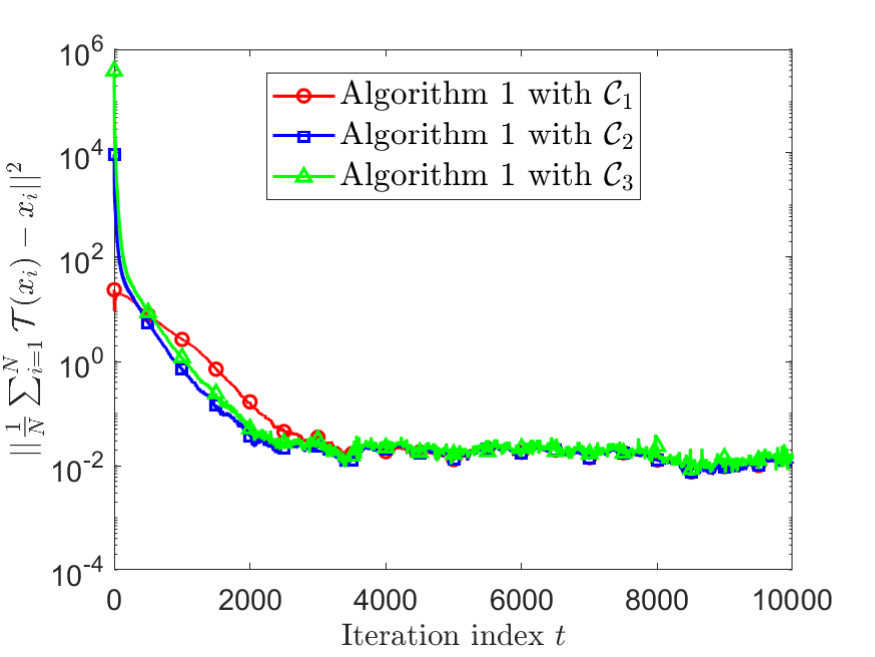}}	
		\quad \quad
		\subfloat[Evolutions with respect to communication bits]{\includegraphics[width=1.5in,height=1.2in]{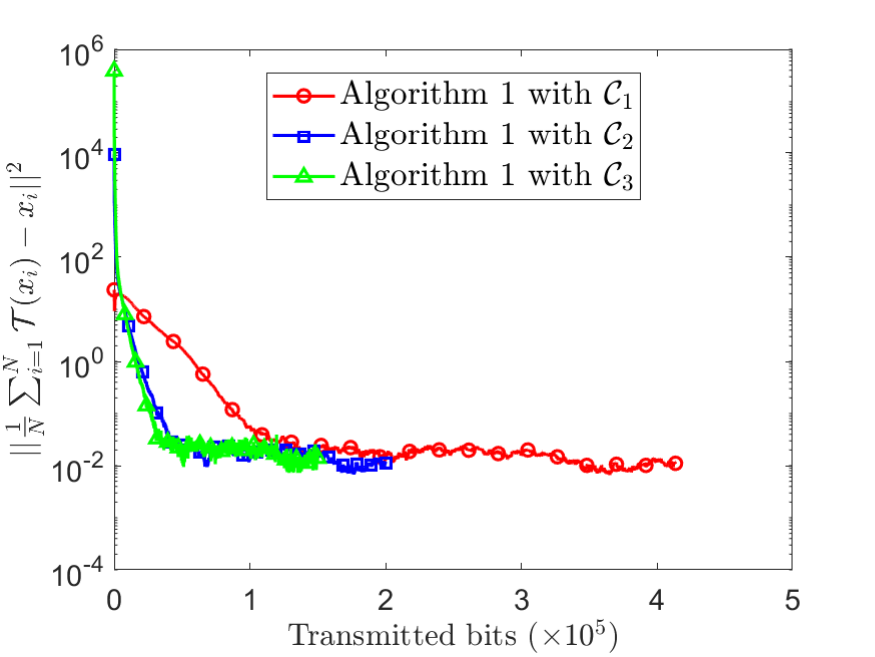}}
		\caption{Evolutions of $\| \frac{1}{N}\sum_{i=1}^{N} \mathcal{T}(x_i)-x_i \|^2$ for Algorithm \ref{alg.1} with different compressors in the non-convex case.}
		\label{FF2}
	\end{figure}
	
	\begin{figure}[!t]
		\centering 	
		\subfloat[Evolutions with respect to the number of iterations]{\includegraphics[width=1.5in,height=1.2in]{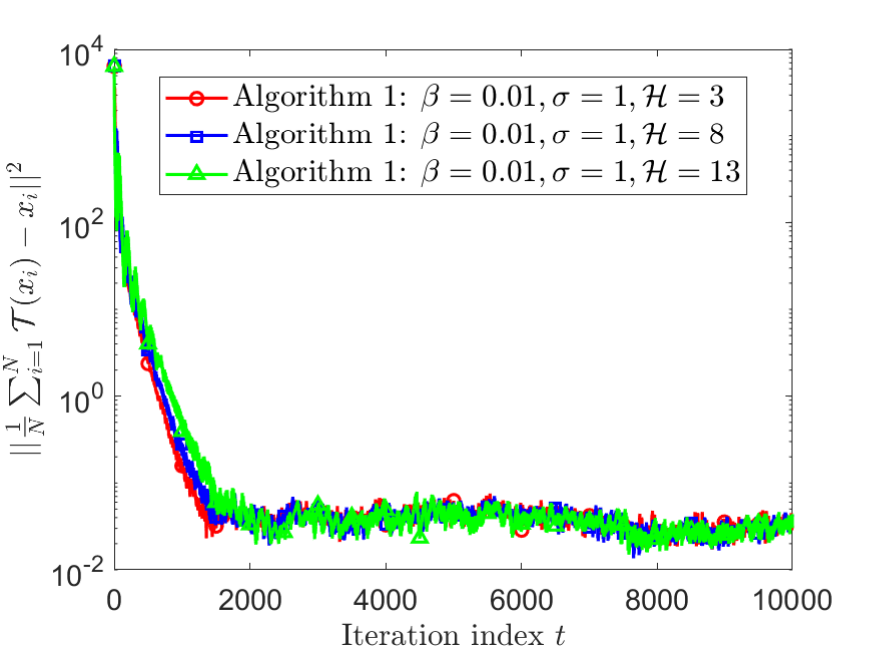}}
		\quad \quad
		\subfloat[Evolutions with respect to communication bits]{\includegraphics[width=1.5in,height=1.2in]{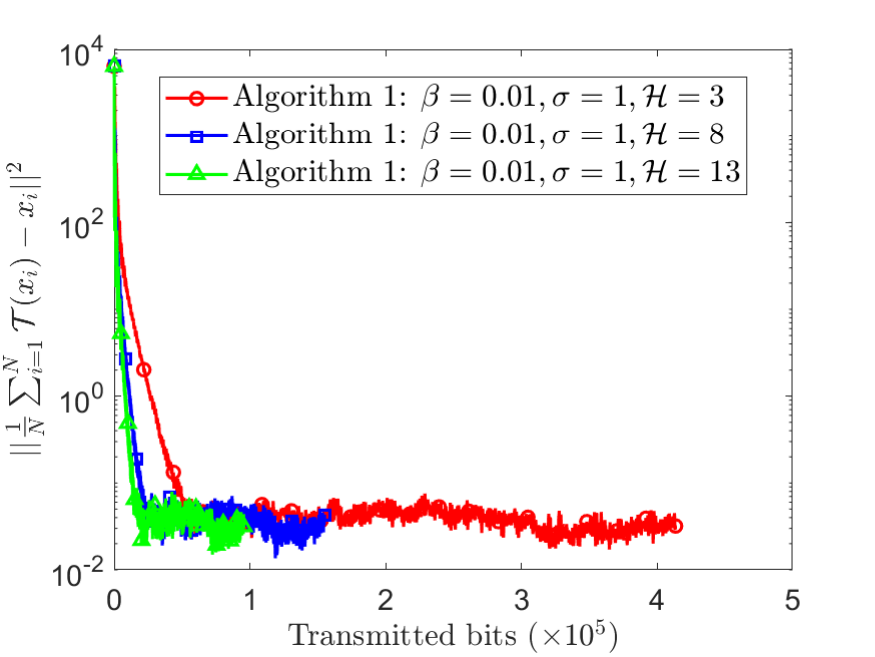}}	
		\caption{Evolutions of $\|\frac{1}{N}\sum_{i=1}^{N} \mathcal{T}(x_i)-x_i \|^2$ for Algorithm~\ref{alg.1} with different communication interval length $\mathcal{H}$ in the non-convex case.}
		\label{FF3}
	\end{figure}
	
	\begin{figure}[!t]
		\centering 	
		\subfloat[Comparisons under different biases settings]{\includegraphics[width=1.5in,height=1.2in]{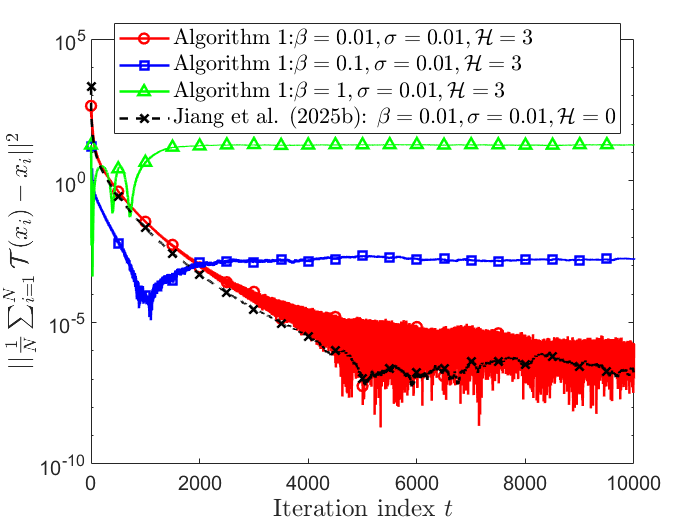}}
		\quad \quad
		\subfloat[Comparisons under different variance settings]{\includegraphics[width=1.5in,height=1.2in]{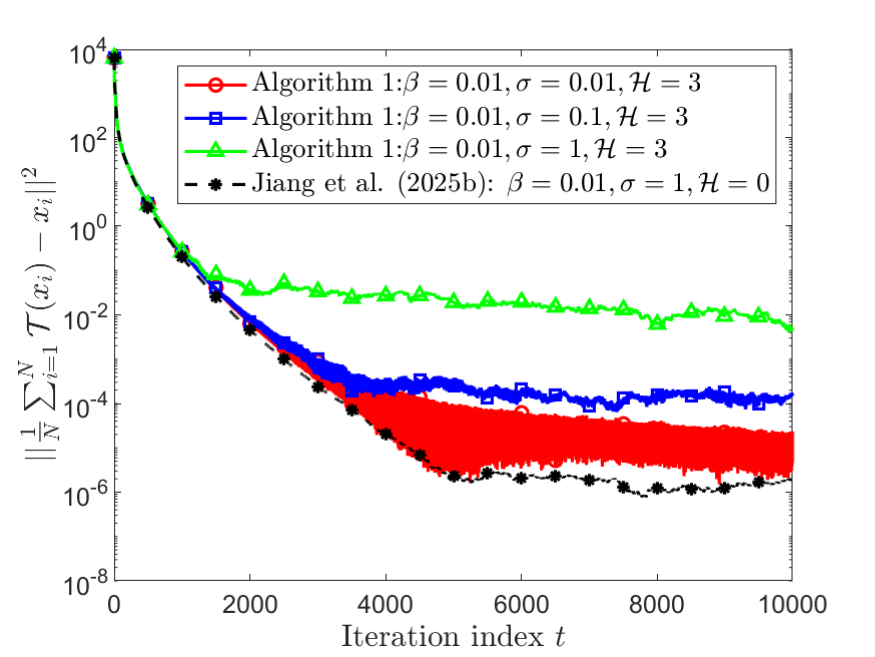}}
		\caption{Evolutions of $\|\frac{1}{N}\sum_{i=1}^{N} \mathcal{T}(x_i)-x_i \|^2$ with different biased stochastic oracles in the non-convex case.}
		\label{FF4}
	\end{figure}

		\vspace{-3mm}
	
	In the simulations, we use a doubly stochastic graph $\mathscr{G}$ generated by the Metropolis--Hastings algorithm~\citep{DDD}. 
	We first compare the performance of Algorithm~\ref{alg.1} against compressed algorithms in~\cite{Pushi,TAC22} using compressor $\mathcal{C}_1$ over  directed graph $\mathscr{G}$. In this case, we consider the typical Gaussian noise $\xi_i^t\sim \mathcal{N}(0,0.1)$.  For Algorithm~\ref{alg.1}, we set $\eta_t=s_t=\frac{0.8}{\sqrt{t+80}}$, $\gamma=0.7$, $\psi=0.99$, $\mathcal{H}=3$.  For the algorithm in~\cite{Pushi}, the parameters are chosen as $\alpha_x = \alpha_y = 0.5$, $\gamma_x = \gamma_y = 0.8$, $\eta = 0.005$. For the algorithm in~\cite{TAC22}, we select $\eta_t=\frac{0.08}{\sqrt{t+80}}$, $\gamma=0.8$, $\mathcal{H}=3$. Running the three algorithms, the evolution of $\frac{1}{N}\sum_{i=1}^{N} \|\mathcal{T}_{i}(x_i)-x_i \|^2$ is shown in Fig.~\ref{FF1}(a). The compressor $\mathcal{C}_1$ is unbiased and exhibits only relative compression error, theoretically ensuring convergence for all three algorithms.  It is observed that the three algorithms achieve similar convergence results when equipped with the same compressor $\mathcal{C}_1$. However Algorithm~\ref{alg.1} demonstrates superior communication efficiency performance to the algorithms in~\cite{Pushi,TAC22} under the same compressor, requiring fewer communication rounds due to its incorporated communication skipping mechanism. Meanwhile, Fig.~\ref{FF1}(b) demonstrates the performance of the three algorithms when using $\mathcal{C}_2$. As can be observed from this figure, the presence of absolute compression error causes the algorithms in~\cite{Pushi,TAC22} to suffer significant performance degradation or even divergence. This underscores the vulnerability of these methods to absolute errors. In contrast, Algorithm~\ref{alg.1} maintains steady convergence, validating our theoretical design where the dynamic scaling parameter $s_t$ successfully suppresses the accumulation of absolute compression errors.

	\vspace{-3mm}
	
	Additionally, with the noise $\xi_{i}^t\sim \mathcal{N}(0,0.01)$ Fig.~\ref{FF2} compares Algorithm \ref{alg.1} across compressors $\mathcal{C}_1$--$\mathcal{C}_3$ from both iteration numbers and communication bits. Notably, $\mathcal{C}_3$ enables Algorithm~\ref{alg.1} to achieve the highest communication efficiency despite introducing both relative and absolute compression errors. This indicates that by leveraging Assumption~\ref{ass3.YA} which offers more flexibility in compressors selection, Algorithm~\ref{alg.1} delivers better performance with fewer communication requirements. To further investigate how the communication interval length $\mathcal{H}$ impacts algorithm performance, Fig.~\ref{FF3} presents numerical results for $\mathcal{H}=3$, $\mathcal{H}=8$ and $\mathcal{H}=13$, evaluated by iteration numbers and communication bits. As seen in the figure, increasing $\mathcal{H}$, which reduces communication frequency, translates to lower communication cost but risks delayed convergence. This behavior reflects a fundamental trade-off between convergence rate and communication cost inherent in selecting $\mathcal{H}$.

		\vspace{-3mm}
	
	Next, we examine scenarios involving biased stochastic oracles, as shown in Fig.~\ref{FF4}, where distinct $\beta$ and $\sigma$ govern specific bias-variance trade-off affecting the rate and neighborhood of the convergence. By setting $\mathcal{H}=3$ and running Algorithm \ref{alg.1} with $\mathcal{C}_2$ and the algorithm in~\cite{BiaseTSP}, Fig.~\ref{FF4} highlights the error reaches the noise level determined by $\beta$ and $\sigma$. 
	As illustrated in Fig.~\ref{FF4}, increasing either the bias $\beta$ or the variance $\sigma$ proportionately enlarges the steady-state error neighborhood. This observation corroborates the theoretical bounds established in Theorem~\ref{Theo1}, where the convergence precision is strictly governed by these stochastic parameters. 
	Notably, the only existing distributed biased stochastic optimization algorithm in~\cite{BiaseTSP}, which adopts a similar biased oracle, is also tested in Fig.~\ref{FF4} with parameters selected as $\eta_t=\frac{0.08}{\sqrt{t+80}}$. As shown in Fig.~\ref{FF4}, Algorithm \ref{alg.1}, equipped with compressor $\mathcal{C}_2$ and a communication skipping mechanism, achieves convergence performance comparable to that of the algorithm in~\cite{BiaseTSP} while demonstrating superior communication efficiency, specifically requiring fewer communication rounds to reach convergence.

			\vspace{-3mm}
	\subsection{Strongly Convex Case}
			\vspace{-3mm}
	We now evaluate the effectiveness of the proposed distributed fixed-point iteration under contractive operators. For consistency, we adopt simulation settings similar to previous experiments, focusing on the distributed stochastic optimization problem (\ref{OppS}). Therein, each component function is strongly convex and is formulated as $f_1\left( x \right) =x^2+1.5x+0.9$, 
	$
	f_2\left( x \right) =0.4x^2+0.7e^x
	$, 
	$
	f_3\left( x \right) =0.2x^4+0.6x^2
	$, 
	$
	f_4\left( x \right) =x^2+1.5x+0.1
	$, 
	$
	f_5\left( x \right) =0.6x^2+0.3e^x
	$, 
	$
	f_6\left( x \right) =0.8x^4+0.4x^2
	$.

	\vspace{-3mm}
		
	As discussed in Section \ref{s2.B}, the corresponding operator $\mathcal{T}_i=Id-\tau \nabla f_i$ with $\tau=0.5$ is contractive. We first consider an unbiased stochastic oracle with $\beta=0$ and $\sigma=0.01$, 
	and we set parameters of Algorithm \ref{alg.1} and the algorithm in~\cite{Pushi} as $\eta_t=s_t=\frac{8}{t+500}$, $\gamma=0.8$, $\psi=0.99$,  $\mathcal{H}=3$ and $\alpha_x = \alpha_y = 0.5$, $\gamma_x = \gamma_y = 0.8$, $\eta = 0.005$, respectively.  
	Fig.~\ref{FF5} illustrates the evolution of the iteration numbers and communication bits for Algorithm~\ref{alg.1} and the algorithm in~\cite{Pushi}, employing compressors $\mathcal{C}_1$--$\mathcal{C}_3$. As depicted in Fig.~\ref{FF5}(a), both algorithms successfully converge under all three compression schemes. Notably, utilizing a constant step size, the algorithm in~\cite{Pushi} exhibits a more rapid initial descent. However, it inevitably oscillates and plateaus at a higher error neighborhood. In contrast, Algorithm~\ref{alg.1}, which employs a decaying step size,  ultimately achieves a more precise steady-state convergence. Furthermore, Fig.~\ref{FF5}(b) demonstrates that compressor $\mathcal{C}_3$ provides the most substantial savings in communication bits for both methods. Ultimately, Algorithm~\ref{alg.1} paired with $\mathcal{C}_3$ delivers the best overall performance, reaching the highest final accuracy with the lowest communication overhead, which fully validates its broad compatibility with general compressors.

		\vspace{-3mm}
	
	Fig.~\ref{FF6} benchmarks Algorithm~\ref{alg.1} against varying bias-variance configurations of stochastic noise for strongly convex objectives. Similar to the non-convex case, we observe that Algorithm~\ref{alg.1} convergence to a neighborhood of the optimal solution, whose size is governed by the bias $\beta$ and the variance $\sigma$. Despite this, it achieves faster convergence rates due to the strong convexity. Crucially, as the bias and variance approach zero, the steady-state error progressively diminishes, confirming our theoretical claim that the proposed framework recovers exact convergence under unbiased conditions.

	%
\begin{figure}[!t]
	\centering 	
	\subfloat[Evolutions with respect to the number of iterations]{\includegraphics[width=1.5in,height=1.2in]{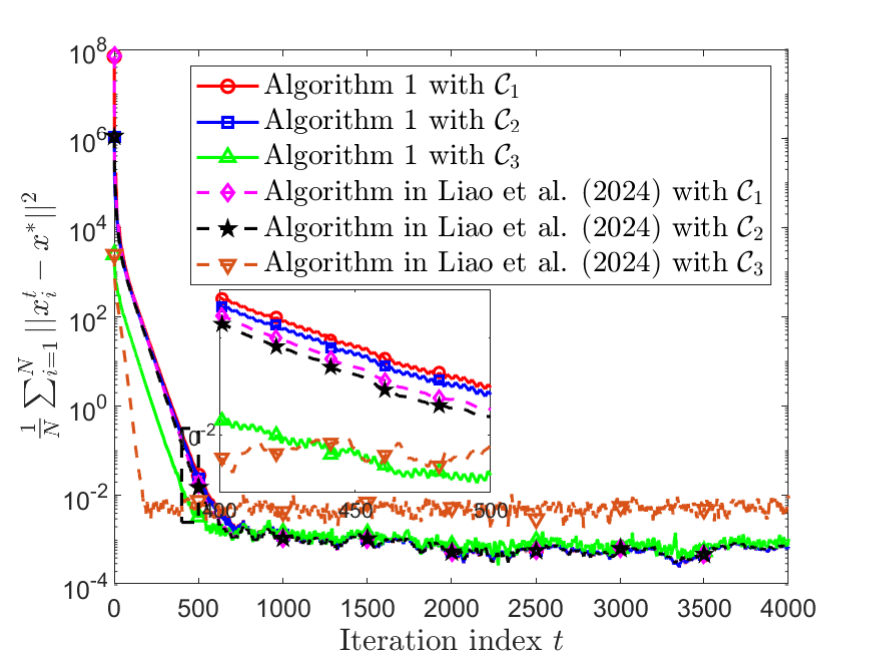}}	
	\quad \quad
	\subfloat[Evolutions with respect to communication bits]{\includegraphics[width=1.5in,height=1.2in]{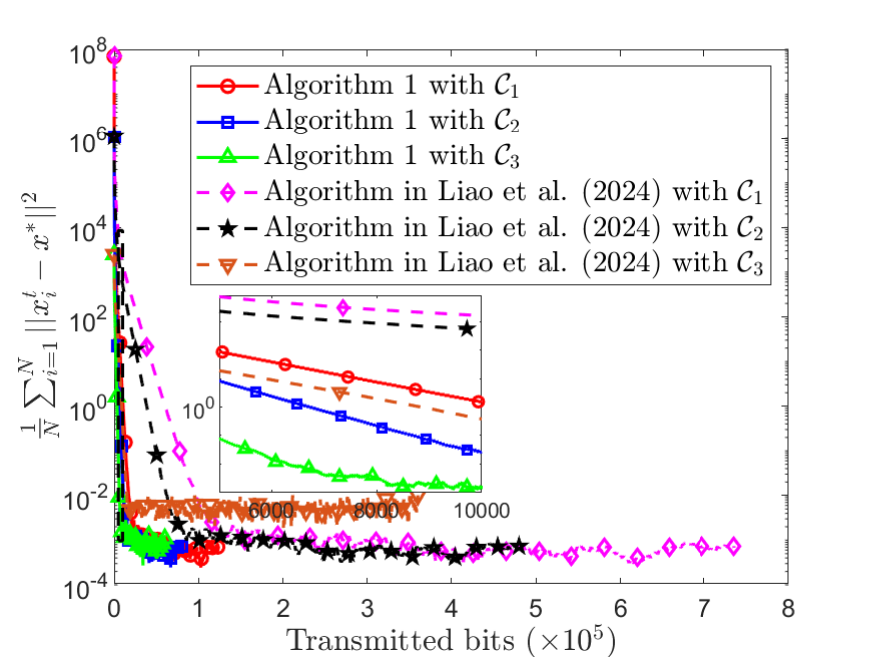}}	
	\caption{Evolutions of $\frac{1}{N} \sum_{i=1}^N \|x_i^t - x^*\|^2 $ for different algorithms equipped with various compressors in the strongly convex case.}
	\label{FF5}
\end{figure}

\begin{figure}[!t]
	\centering 	
	\subfloat[Comparisons under different biases settings]{\includegraphics[width=1.5in,height=1.2in]{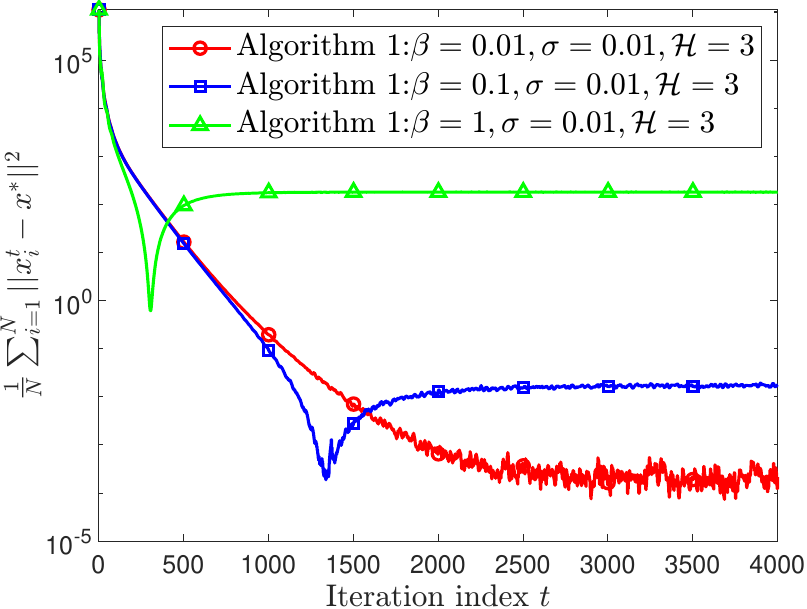}}
	\quad \quad
	\subfloat[Comparisons under different variance settings]{\includegraphics[width=1.5in,height=1.2in]{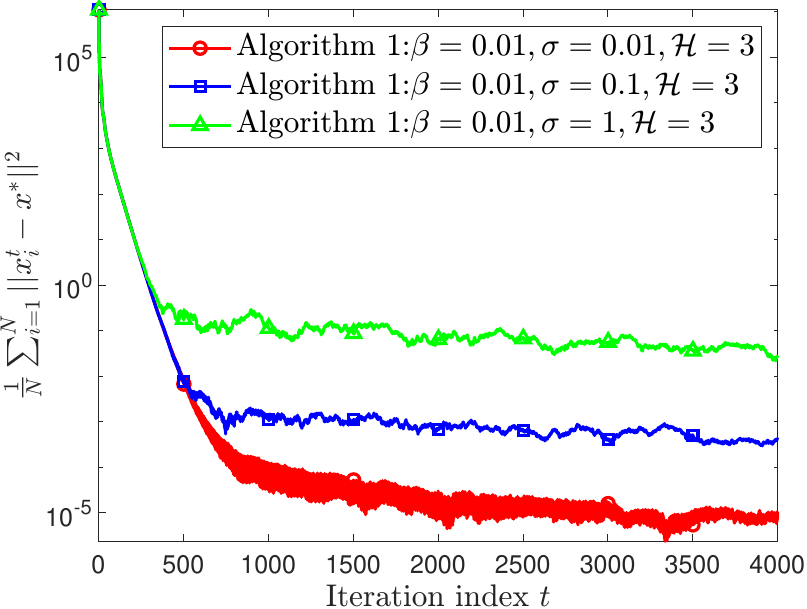}}
	\caption{Evolutions of $\frac{1}{N} \sum_{i=1}^N \|x_i^t - x^*\|^2 $ with different biased stochastic oracles in the strongly convex case.}
	\label{FF6}
\end{figure}

			\vspace{-4mm}
	\section{Conclusion}\label{ScCon}
			\vspace{-3mm}
	In this paper, we focused on the distributed fixed point seeking 
	problem for sum-separable operators. We proposed a distributed algorithm with communication compression and dynamic period skipping mechanisms. The proposed  framework advances the theoretical limits of distributed fixed  point iterations by supporting expansive operators and biased 
	stochastic settings. Moreover, we generalized the convergence analysis  to accommodate relaxed growth bias and variance conditions, 
	bridging the gap between centralized and distributed settings. 
	Finally, the effectiveness of the algorithm was illustrated by numerical simulations. Future work will focus on enhancing the computational efficiency of accelerated fixed-point iterations and developing online learning frameworks for time-varying operators.

	\ifbool{submitA}{
	}{\vspace{-3mm}
\section{Appendix}
	\vspace{-3mm}

	To facilitate subsequent analysis, we reformulate the algorithm into a compact form. For this purpose, some stacked variables and their averages are introduced below 		\vspace{-3mm}
	\begin{align*}
		& \mathbf{X}^t=\left[x_1^t,\ldots,x_N^t\right]^{\top} \in \mathbb{R}^{N \times n}, \widehat{\mathbf{X}}^t=\left[\hat{x}_1^t, \ldots, \hat{x}_N^t\right]^{\top} \in \mathbb{R}^{N \times n}, \\
		& \bar{{x}}^t=\frac{1}{N} \sum_{i=1}^N {x}_i^t \in \mathbb{R}^n, \overline{\mathbf{X}}^t=\left[\bar{x}^t, \cdots, \bar{x}^t\right]^\top \in \mathbb{R}^{N \times n}, \\
		&\mathbf{Z}^t=\left[z_1^t,\ldots,z_N^t\right]^{\top} \in \mathbb{R}^{N \times n},
		\bar{{z}}^t=\frac{1}{N} \sum_{i=1}^N {z}_i^t \in \mathbb{R}^n,\\ &\overline{\mathbf{Z}}^t=\left[\bar{z}^t, \cdots, \bar{z}^t\right]^\top \in \mathbb{R}^{N \times n}, \\
		& \widetilde{\mathcal{T}}(\mathbf{X}^t,\boldsymbol{\xi}^t)=\left[\widetilde{\mathcal{T}}_1\left(x_1^t, \xi_1^t\right), \ldots, \widetilde{\mathcal{T}}_N\left(x_N^t,\xi_N^t\right)\right]^{\top} \in \mathbb{R}^{N \times n}.
	\end{align*}
			\vspace{-7mm}
	
	Based on the above notations, consider Algorithm~\ref{alg.1} with indices given by $\mathcal{I}_{T}=\{\mathcal{I}_{(1)},\mathcal{I}_{(2)},\ldots,\mathcal{I}_{(k)},\ldots\}$. Then the updates from index $\mathcal{I}_{(t)}$ to $\mathcal{I}_{(t+1)}$ can be written in the following compact form		\vspace{-3mm}
	\begin{subequations}\label{jinchou}
		\begin{align}
			\mathbf{Z}^{\mathcal{I}_{(t+1)}}=&\mathbf{X}^{\mathcal{I}_{ (t) }}-\sum_{t'=\mathcal{I}_{(t)}}^{\mathcal{I}_{(t+1)}-1}{\eta_{t'}}({\mathbf{X}^{t'}}
			- {\widetilde{\mathcal{T}}( \mathbf{X}^{ t' },\boldsymbol{\xi }^{t'})}), \label{aa1}\\
			\mathbf{X}^{\mathcal{I}_{(t+1)  }}=&\mathbf{Z}^{\mathcal{I}_{(t+1)}}+\gamma ( \mathbf{W}-\mathbf{I}) \widehat{\mathbf{X}}^{\mathcal{I}_{(t)}}\label{aa3},\\
			\widehat{\mathbf{X}}^{\mathcal{I}_{(t+1)}}=&\widehat{\mathbf{X}}^{\mathcal{I}_{(t)}}+\psi{s}_{ \mathcal{I}_{(t)} }\mathcal{C}( ( \mathbf{X}^{\mathcal{I}_{ (t+1) }}-\widehat{\mathbf{X}}^{ \mathcal{I}_{(t)} } ) /{s}_{ \mathcal{I}_{(t)} } ).\label{aa2}
		\end{align}
	\end{subequations}
			\vspace{-7mm}

			\vspace{-3mm}
	\subsection{Useful Lemmas}		\vspace{-3mm}
	In this section, we provide some supporting lemmas and recurrence results which will be used in the convergence analysis of Algorithm \ref{alg.1}.
	
	We now introduce a lemma that bounds the expected norm of the iterates.
	\begin{lemma}\label{lee.2}
		Under Assumptions \ref{ass5.Un} and \ref{ass4.YJ}, for any $t\ge 1$, the sequence $\{\mathbf{X}^{\mathcal{I}_{(t)}}\}$ of Algorithm~\ref{alg.1} satisfies 
		\begin{align*}
			\mathbb{E}\left[\left\|\mathbf{Z}^{\mathcal{I}_{(t+1)}}-\mathbf{X}^{\mathcal{I}_{(t)}}\right\|_{\text{F}}^2\right]\leq& N  \mathcal{D}^2\mathcal{H}^2\eta_{\mathcal{I}_{(t)}}^2,\\ \mathbb{E}\left[\left\|\overline{\mathbf{Z}}^{\mathcal{I}_{(t+1)}}-\overline{\mathbf{X}}^{\mathcal{I}_{(t)}}\right\|_{\text{F}}^2\right]\leq&  {N\mathcal{D}^2\mathcal{H}^2\eta_{\mathcal{I}_{(t)}}^2},\\ 
			\mathbb{E}\left[ \left\|\overline{\mathbf{X}}^{\mathcal{I}_{(t+1)}}-\overline{\mathbf{X}}^{\mathcal{I}_{(t)}}\right\|_{\text{F}}^2\right]\leq&  {N\mathcal{D}^2\mathcal{H}^2\eta_{\mathcal{I}_{(t)}}^2}.
		\end{align*}
	\end{lemma}
		\vspace{-7mm}
		\begin{proof}
			First, invoking (\ref{aa1}) and Assumption \ref{ass4.YJ} yields
			\begin{align}\label{K1}
				&\mathbb{E}\left[\left\|\mathbf{Z}^{\mathcal{I}_{(t+1)}}-\mathbf{X}^{\mathcal{I}_{(t)}}\right\|_{\text{F}}^2\right]\nonumber\\
				\le&\sum_{i=1}^N \eta_{\mathcal{I}_{(t)}}^2 \mathbb{E}\left[\left\|\sum_{t'=\mathcal{I}_{(t)}}^{\mathcal{I}_{(t+1)}-1} {x_{i}^{t'}}- \widetilde{\mathcal{T}}_i\left( x_{i}^{t'},\xi _{i}^{t'} \right)\right\|^2\right] \nonumber\\
				\textstyle\leq& N  \mathcal{D}^2\mathcal{H}^2\eta_{\mathcal{I}_{(t)}}^2.
			\end{align}
			Moreover, by the definition of $\overline{\mathbf{X}}^t$ and (\ref{aa1}), we get
			\begin{align}\label{K2}
				&{ \mathbb{E}\left[\left\|\overline{\mathbf{Z}}^{\mathcal{I}_{(t+1)}}-\overline{\mathbf{X}}^{\mathcal{I}_{(t)}}\right\|_{\text{F}}^2\right]} \nonumber \\
				& =  N\mathbb{E}\left[\left\|\frac{1}{N}\sum_{i=1}^N {z}_i^{\mathcal{I}_{(t+1)}}-\frac{1}{N}\sum_{i=1}^N {x}_i^{\mathcal{I}_{(t)}}\right\|^2\right]  \nonumber \\
				& \le{\eta_{\mathcal{I}_{(t)}}^2}\mathbb{E}\left[\sum_{i=1}^{N} \left\|\sum_{t'=\mathcal{I}_{(t)}}^{\mathcal{I}_{(t+1)}-1} x_{i}^{t'}- \widetilde{\mathcal{T}}_i\left( x_{i}^{t'},\xi _{i}^{t'} \right)\right\|^2\right] \nonumber \\
				&\leq {N\mathcal{D}^2\mathcal{H}^2\eta_{\mathcal{I}_{(t)}}^2} ,
			\end{align}
			where the second inequality holds since the Cauchy--Schwarz inequality. 
			In addition, applying $\mathbf{1}_N^\top(\mathbf{W}-\mathbf{I})=\mathbf{0}^\top$ and (\ref{aa3}) implies
			\begin{align}\label{K3}
				\mathbb{E}\left[\left\|\overline{\mathbf{X}}^{\mathcal{I}_{(t+1)}}-\overline{\mathbf{X}}^{\mathcal{I}_{(t)}}\right\|_{\text{F}}^2 \right]\leq {N\mathcal{D}^2\mathcal{H}^2\eta_{\mathcal{I}_{(t)}}^2}. 
			\end{align}
			
			Combining the above analysis completes the proof.
		\end{proof}
	 		\vspace{-3mm}

	Next, we present a lemma that derives a recursive inequality for the consensus error. 
	\begin{lemma}\label{lee.3} 		\vspace{-3mm}
		Under Assumption \ref{ass5.Un}, for any $t\ge1$, there exist some positive constants $\mu_i, i \in \{1,2,3\}$, such that the sequence $\{\mathbf{X}^{\mathcal{I}_{(t)}}\}$ generated by Algorithm \ref{alg.1} satisfies
		\begin{align}\label{LL0}
			&\mathbb{E}\left[\left\|\mathbf{X}^{\mathcal{I}_{(t+1)}}-\overline{\mathbf{X}}^{\mathcal{I}_{(t+1)}}\right\|_{\text{F}}^2\right] \nonumber\\
			&\leq \omega_1\mathbb{E}\left[\left\|\widehat{\mathbf{X}}^{\mathcal{I}_{(t)}}-\mathbf{X}^{\mathcal{I}_{(t)}}\right\|_{\text{F}}^2\right] 
			+\omega_2\mathbb{E}\left[\left\|\mathbf{X}^{\mathcal{I}_{(t)}}-\overline{\mathbf{X}}^{\mathcal{I}_{(t)}}\right\|_{\text{F}}^2\right] \nonumber\\
			&\quad +\omega_3 N \mathcal{D}^2\mathcal{H}^2\eta_{\mathcal{I}_{(t)}}^2,
		\end{align}
		where $\omega_1 := (1+\mu_1)(1+\mu_2)\gamma^2 {\alpha}^2$, $\omega_2 := (1+\mu_1^{-1})(1+\mu_3^{-1})(1-\kappa \gamma)^2$ and $\omega_3 := (1+\mu_1)(1+\mu_2^{-1})\gamma^2 {\alpha}^2 + 3(1+\mu_1^{-1})(1+\mu_3)[(1-\kappa \gamma)^2 + (1-\kappa \gamma +\gamma)^2 + \gamma^2]$.
	\end{lemma}
		\begin{proof} By (\ref{aa3}) and taking expectations, simple calculations give rise to 
			\begin{align}\label{LL1}
				& \mathbb{E}\Big[\big\|\mathbf{X}^{\mathcal{I}_{(t+1)}}-\overline{\mathbf{X}}^{\mathcal{I}_{(t+1)}}\big\|_{\text{F}}^2\Big] \nonumber\\
				& =\mathbb{E}\Big[\big\|\mathbf{Z}^{\mathcal{I}_{(t+1)}}+\gamma(\mathbf{W}-\mathbf{I}) \widehat{\mathbf{X}}^{\mathcal{I}_{(t)}}-\overline{\mathbf{X}}^{\mathcal{I}_{(t+1)}}\big\|_{\text{F}}^2\Big] \nonumber\\
				& =\mathbb{E}\bigg[\Big\|\gamma(\mathbf{W}-\mathbf{I})\big(\widehat{\mathbf{X}}^{\mathcal{I}_{(t)}}-\mathbf{Z}^{\mathcal{I}_{(t+1)}}\big)\nonumber\\
				&\quad \ \ +(\gamma \mathbf{W}+(1-\gamma) \mathbf{I})\big(\mathbf{Z}^{\mathcal{I}_{(t+1)}}-\overline{\mathbf{X}}^{\mathcal{I}_{(t+1)}}\big)\Big\|_{\text{F}}^2\bigg] \nonumber\\
				& \leq (1+\mu_1) \gamma^2\mathbb{E}\Big[\big\|(\mathbf{W}-\mathbf{I})\big(\widehat{\mathbf{X}}^{\mathcal{I}_{(t)}}-\mathbf{Z}^{\mathcal{I}_{(t+1)}}\big)\big\|_{\text{F}}^2\Big] \nonumber\\
				& \quad +(1+\mu_1^{-1})\mathbb{E}\Big[\big\|(\gamma \mathbf{W}+(1-\gamma) \mathbf{I})\big(\mathbf{Z}^{\mathcal{I}_{(t+1)}}-\overline{\mathbf{X}}^{\mathcal{I}_{(t+1)}}\big)\big\|_{\text{F}}^2\Big]\nonumber\\
				& \leq (1+\mu_1) \gamma^2 \alpha^2\mathbb{E}\Big[\big\|\widehat{\mathbf{X}}^{\mathcal{I}_{(t)}}-\mathbf{Z}^{\mathcal{I}_{(t+1)}}\big\|_{\text{F}}^2\Big]\nonumber\\
				& \quad +(1+\mu_1^{-1})\mathbb{E}\Big[\big\|(\gamma \mathbf{W}+(1-\gamma) \mathbf{I})\big(\mathbf{Z}^{\mathcal{I}_{(t+1)}}-\overline{\mathbf{X}}^{\mathcal{I}_{(t+1)}}\big)\big\|_{\text{F}}^2\Big].
			\end{align}
			where the second equality follows from $(\mathbf{W}-\mathbf{I})\overline{\mathbf{X}}^{\mathcal{I}_{(t)}}=\mathbf{0}$; the first inequality invokes $\|\mathbf{A}_1+\mathbf{A}_2\|_{\text{F}}^2 \leq (1+\mu_1)\|\mathbf{A}_1\|_{\text{F}}^2+(1+\mu_1^{-1})\|\mathbf{A}_2\|_{\text{F}}^2$ for $\mu_1>0$; and the second inequality is due to $\|\mathbf{A}_1 \mathbf{A}_2\|_{\text{F}} \leq\|\mathbf{A}_1\|_2\|\mathbf{A}_2\|_{\text{F}}$ and $\alpha=\|\mathbf{W}-\mathbf{I}\|_2$.

			In the sequel, for each term in (\ref{LL1}), we have
			\begin{align}\label{LL2}
				&\mathbb{E}\left[\left\|\widehat{\mathbf{X}}^{\mathcal{I}_{(t)}}-\mathbf{Z}^{\mathcal{I}_{(t+1)}}\right\|_{\text{F}}^2\right]
				\le \left(1+\mu_2\right)\mathbb{E}\left[\left\|\widehat{\mathbf{X}}^{\mathcal{I}_{(t)}}-\mathbf{X}^{\mathcal{I}_{(t)}}\right\|_{\text{F}}^2\right] \nonumber\\
				&+\left(1+\mu_2^{-1}\right)\mathbb{E}\left[\left\|\mathbf{X}^{\mathcal{I}_{(t)}}-\mathbf{Z}^{\mathcal{I}_{(t+1)}}\right\|_{\text{F}}^2\right].
			\end{align}
			Further, consider the term inside the expectation of the second part in (\ref{LL1}), we have
			\begin{align}\label{LL3}
				&\left\|(\gamma \mathbf{W}+(1-\gamma) \mathbf{I})\left(\mathbf{Z}^{\mathcal{I}_{(t+1)}}-\overline{\mathbf{X}}^{\mathcal{I}_{(t+1)}}\right)\right\|_{\text{F}} \nonumber\\
				\leq&(1-\gamma)\left\|\mathbf{Z}^{\mathcal{I}_{(t+1)}}-\overline{\mathbf{X}}^{\mathcal{I}_{(t+1)}}\right\|_{\text{F}}\nonumber\\
				&+\gamma\left\|\mathbf{W}-\frac{1}{N}\mathbf{1}_N\mathbf{1}_N^\top\right\|_{2}\left\|\mathbf{Z}^{\mathcal{I}_{(t+1)}}-\overline{\mathbf{X}}^{\mathcal{I}_{(t+1)}}\right\|_{\text{F}}\nonumber\\
				& +\gamma\left\|\overline{\mathbf{Z}}^{\mathcal{I}_{(t+1)}}-\overline{\mathbf{X}}^{\mathcal{I}_{(t+1)}}\right\|_{\text{F}} \nonumber\\
				{\le}&(1-\kappa \gamma)\left\|\mathbf{Z}^{\mathcal{I}_{(t+1)}}-\overline{\mathbf{X}}^{\mathcal{I}_{(t+1)}}\right\|_{\text{F}}+\gamma\left\|\overline{\mathbf{Z}}^{\mathcal{I}_{(t+1)}}-\overline{\mathbf{X}}^{\mathcal{I}_{(t+1)}}\right\|_{\text{F}}\nonumber\\
				{\le}&(1-\kappa \gamma)\left\|\mathbf{X}^{\mathcal{I}_{(t)}}-\overline{\mathbf{X}}^{\mathcal{I}_{(t)}}\right\|_{\text{F}}+(1-\kappa \gamma)\left\|\mathbf{Z}^{\mathcal{I}_{(t+1)}}-\mathbf{X}^{\mathcal{I}_{(t)}}\right\|_{\text{F}}\nonumber \\
				&+(1-\kappa\gamma+\gamma)\left\|\overline{\mathbf{X}}^{\mathcal{I}_{(t)}}-\overline{\mathbf{X}}^{\mathcal{I}_{(t+1)}}\right\|_{\text{F}}\nonumber\\
				&+\gamma\left\|\overline{\mathbf{Z}}^{\mathcal{I}_{(t+1)}}-\overline{\mathbf{X}}^{\mathcal{I}_{(t)}}\right\|_{\text{F}}, 
			\end{align}
			where the first inequality is due to $\frac{1}{N}\mathbf{1}_N\mathbf{1}_N^\top \mathbf{Z}^{\mathcal{I}_{(t+1)}} =\overline{\mathbf{Z}}^{\mathcal{I}_{(t+1)}}$. 
			Building on it, taking the square and then the expectation, one can further obtain that
			\begin{align}\label{LL4}
				& \mathbb{E}\left[\left\|(\gamma \mathbf{W}+(1-\gamma) \mathbf{I})\left(\mathbf{Z}^{\mathcal{I}_{(t+1)}}-\overline{\mathbf{X}}^{\mathcal{I}_{(t+1)}}\right)\right\|_{\text{F}}^2\right]\nonumber \\
				\le&\left(1+\mu_3^{-1}\right)(1-\kappa \gamma)^2 \mathbb{E}\left[\left\|\mathbf{X}^{\mathcal{I}_{(t)}}-\overline{\mathbf{X}}^{\mathcal{I}_{(t)}}\right\|_{\text{F}}^2\right] \nonumber \\
				&+(1+\mu_3)\mathbb{E}\left[\left((1-\kappa \gamma)\left\|\mathbf{Z}^{\mathcal{I}_{(t+1)}}-\mathbf{X}^{\mathcal{I}_{(t)}}\right\|_{\text{F}}\nonumber +(1-\kappa\gamma+\gamma)\cdot\right.\right. \\
				&\left.\left.\left\|\overline{\mathbf{X}}^{\mathcal{I}_{(t)}}-\overline{\mathbf{X}}^{\mathcal{I}_{(t+1)}}\right\|_{\text{F}}+\gamma\left\|\overline{\mathbf{Z}}^{\mathcal{I}_{(t+1)}}-\overline{\mathbf{X}}^{\mathcal{I}_{(t)}}\right\|_{\text{F}} \right)^2\right]\nonumber \\
				\le&(1+\mu_3^{-1})(1-\kappa \gamma)^2 \mathbb{E}\left[\left\|\mathbf{X}^{\mathcal{I}_{(t)}}-\overline{\mathbf{X}}^{\mathcal{I}_{(t)}}\right\|_{\text{F}}^2\right] \nonumber \\
				&+3\left(1+\mu_3\right)(1-\kappa \gamma)^2 \mathbb{E}\left[\left\|\mathbf{Z}^{\mathcal{I}_{(t+1)}}-\mathbf{X}^{\mathcal{I}_{(t)}}\right\|_{\text{F}}^2\right] \nonumber \\
				&+3\left(1+\mu_3\right)(1-\kappa \gamma+\gamma)^2 \mathbb{E}\left[\left\|\overline{\mathbf{X}}^{\mathcal{I}_{(t)}}-\overline{\mathbf{X}}^{\mathcal{I}_{(t+1)}}\right\|_{\text{F}}^2\right] \nonumber \\
				&+3\left(1+\mu_3\right) \gamma^2 \mathbb{E}\left[\left\|\overline{\mathbf{Z}}^{\mathcal{I}_{(t+1)}}-\overline{\mathbf{X}}^{\mathcal{I}_{(t)}}\right\|_{\text{F}}^2\right]\nonumber\\
				\le&\left(1+\mu_3^{-1}\right)(1-\kappa \gamma)^2 \mathbb{E}\left[\left\|\mathbf{X}^{\mathcal{I}_{(t)}}-\overline{\mathbf{X}}^{\mathcal{I}_{(t)}}\right\|_{\text{F}}^2\right] \nonumber \\
				&+3\left(1+\mu_3\right)(1-\kappa \gamma)^2 N  \mathcal{D}^2\mathcal{H}^2\eta_{\mathcal{I}_{(t)}}^2 \nonumber \\
				&+{3\left(1+\mu_3\right)(1-\kappa \gamma+\gamma)^2 }
				N\mathcal{D}^2\mathcal{H}^2\eta_{\mathcal{I}_{(t)}}^2 \nonumber \\
				&+{ 3\left(1+\mu_3\right) \gamma^2  }N\mathcal{D}^2\mathcal{H}^2\eta_{\mathcal{I}_{(t)}}^2,
			\end{align}
			where the first inequality invokes $\left \|\mathbf{A}_1+\mathbf{A}_2 \right \|_{\text{F}}^2\le (1+\mu^{-1})\left \|\mathbf{A}_1\right \|_{\text{F}}^2+(1+\mu)\left \|\mathbf{A}_2\right\|_{\text{F}}^2$; the second inequality is due to $(a+b+c)^2\le 3a^2+3b^2+3c^2$ and the third inequality is deduced by Lemma \ref{lee.2}. 
			Plugging (\ref{LL2}) and (\ref{LL4}) into (\ref{LL1}), it is clear that (\ref{LL0}) holds. 
	\end{proof}

	To address the impact of compressed communication, we provide the following recursive inequality.
	\begin{lemma}\label{lee.4}
		Under Assumptions \ref{ass5.Un}, \ref{ass2.YY} and \ref{ass4.YJ}, for any $t\ge 1$, let the sequence $\{\mathbf{X}^{\mathcal{I}_{(t)}}\}$ be generated by Algorithm \ref{alg.1} with parameter $\psi \in (\frac{3}{4r},\frac{1}{r}]$, then there exist constants $\mu_i>0, i\in \{4,5,6\}$ such that the following recursion holds
		\begin{align}\label{L4}
			&\mathbb{E}\left[\left\|\mathbf{X}^{\mathcal{I}_{(t+1)}}-\widehat{\mathbf{X}}^{\mathcal{I}_{(t+1)}}\right\|_{\text{F}}^2\right] \nonumber\\
			&\leq \omega_4 \mathbb{E}\left[\left\|\mathbf{X}^{\mathcal{I}_{(t)}}-\widehat{\mathbf{X}}^{\mathcal{I}_{(t)}}\right\|_{\text{F}}^2\right]
			+\omega_5\mathbb{E}\left[\left\|\mathbf{X}^{\mathcal{I}_{(t)}}-\overline{\mathbf{X}}^{\mathcal{I}_{(t)}}\right\|_{\text{F}}^2\right] \nonumber \\
			&\quad +\omega_6 N \mathcal{D}^2\mathcal{H}^2\eta_{\mathcal{I}_{(t)}}^2 + Ns_{\mathcal{I}_{(t)}}^2\psi r\delta^2,
		\end{align}
		where $\omega_4 := (1-\psi r\varphi)(1+\mu_4)(1+\mu_5^{-1})(1+\gamma\alpha)^2$, $\omega_5 := (1-\psi r\varphi)(1+\mu_4^{-1})(1+\mu_6^{-1})\gamma^2\alpha^2$ and $\omega_6 := (1-\psi r\varphi)[(1+\mu_4)(1+\mu_5)(1+\gamma\alpha)^2 + (1+\mu_4^{-1})(1+\mu_6)\gamma^2\alpha^2]$.
	\end{lemma}
		\begin{proof} According to (\ref{aa2}), letting $\Delta \mathbf{x}_i = ({x}_i^{\mathcal{I}_{(t+1)}}-\widehat{{x}}_i^{\mathcal{I}_{(t)}})/s_{\mathcal{I}_{(t)}}$ for notation brevity, we have 
			\begin{align}\label{Ya}
				&\mathbb{E}\left[\left\|\mathbf{X}^{\mathcal{I}_{(t+1)}}-\widehat{\mathbf{X}}^{\mathcal{I}_{(t+1)}}\right\|_{\text{F}}^2\right] \nonumber\\
				& = \sum_{i=1}^N s_{\mathcal{I}_{(t)}}^2 \mathbb{E}\left[\left\| \Delta \mathbf{x}_i - \psi \mathcal{C}(\Delta \mathbf{x}_i) \right\|^2\right] \nonumber\\
				& = \sum_{i=1}^N s_{\mathcal{I}_{(t)}}^2 \mathbb{E}\left[\left\| (1-\psi r)\Delta \mathbf{x}_i + \psi r \left( \Delta \mathbf{x}_i - \frac{\mathcal{C}(\Delta \mathbf{x}_i)}{r} \right) \right\|^2\right] \nonumber\\
				&\le \sum_{i=1}^N s_{\mathcal{I}_{(t)}}^2 \left( (1-\psi r)\|\Delta \mathbf{x}_i\|^2 + \psi r \mathbb{E}\left[\left\| \Delta \mathbf{x}_i - \frac{\mathcal{C}(\Delta \mathbf{x}_i)}{r} \right\|^2\right] \right) \nonumber\\
				&\le \sum_{i=1}^N s_{\mathcal{I}_{(t)}}^2 \Big( (1-\psi r)\|\Delta \mathbf{x}_i\|^2 + \psi r \big( (1-\varphi)\|\Delta \mathbf{x}_i\|^2 + \delta^2 \big) \Big) \nonumber\\
				& = (1-\psi r\varphi)\left\|\mathbf{X}^{\mathcal{I}_{(t+1)}} - \widehat{\mathbf{X}}^{\mathcal{I}_{(t)}}\right\|_{\text{F}}^2 + N\psi r s_{\mathcal{I}_{(t)}}^2\delta^2,
			\end{align}	
			where  the second inequality is due to Assumption \ref{ass3.YA}.
			
			To move forward, from (\ref{aa3}) we have 
			\begin{align}\label{Ya2}
				&\mathbb{E}\left[\left\|{\mathbf{X}}^{\mathcal{I}_{(t+1)}}-\widehat{\mathbf{X}}^{\mathcal{I}_{(t)}}\right\|_{\text{F}}^2\right]\nonumber\\
				=& \mathbb{E}\Bigg[\bigg\| ((1+\gamma) \mathbf{I}-\gamma \mathbf{W}) \left(\mathbf{Z}^{\mathcal{I}_{(t+1)}}-\widehat{\mathbf{X}}^{\mathcal{I}_{(t)}}\right) \nonumber\\
				&\quad +\gamma(\mathbf{W}-\mathbf{I})\left(\mathbf{Z}^{\mathcal{I}_{(t+1)}}-\overline{\mathbf{X}}^{\mathcal{I}_{(t)}}\right) \bigg\|_{\text{F}}^2\Bigg]\nonumber \\
				\le&\left(1+\mu_4\right) \|(1+\gamma) \mathbf{I}-\gamma \mathbf{W}\|_2^2 \mathbb{E}\left[\left\|\mathbf{Z}^{\mathcal{I}_{(t+1)}}-\widehat{\mathbf{X}}^{\mathcal{I}_{(t)}}\right\|_{\text{F}}^2\right] \nonumber\\
				& +\left(1+{\mu_4}^{-1}\right) \gamma^2 \|\mathbf{W}-\mathbf{I}\|_2^2 \mathbb{E}\left[\left\|\mathbf{Z}^{\mathcal{I}_{(t+1)}}-\overline{\mathbf{X}}^{\mathcal{I}_{(t)}}\right\|_{\text{F}}^2\right]\nonumber \\
				\le & \left(1+\mu_4\right)(1+\gamma \alpha)^2 \mathbb{E}\left[\left\|\mathbf{Z}^{\mathcal{I}_{(t+1)}}-\widehat{\mathbf{X}}^{\mathcal{I}_{(t)}}\right\|_{\text{F}}^2\right] \nonumber\\
				& +\left(1+{\mu_4}^{-1}\right) \gamma^2 \alpha^2 \mathbb{E}\left[\left\|\mathbf{Z}^{\mathcal{I}_{(t+1)}}-\overline{\mathbf{X}}^{\mathcal{I}_{(t)}}\right\|_{\text{F}}^2\right],
			\end{align}
			where the equality is due to $(\mathbf{W}-\mathbf{I})\overline{\mathbf{X}}^{\mathcal{I}_{(t)}}=\mathbf{0}$  and the first inequality is deduced by the fact that $\left\| \mathbf{A}_1+ \mathbf{A}_2\right\|_{\text{F}}^2 \le (1+\mu)\left\| \mathbf{A}_1\right\|_{\text{F}}^2+(1+\mu^{-1})\left\| \mathbf{A}_2\right\|_{\text{F}}^2$ for any matrices $\mathbf{A}_1, \mathbf{A}_2$ and $\mu>0$.
			
			Then, separating the expected terms in (\ref{Ya2}) by Young's inequality, we obtain
			\begin{align}
				&\mathbb{E}\left[\left\|\mathbf{Z}^{\mathcal{I}_{(t+1)}}-\widehat{\mathbf{X}}^{\mathcal{I}_{(t)}}\right\|_{\text{F}}^2\right]
				\le(1+\mu_5) \mathbb{E} \left[\left\|\mathbf{Z}^{\mathcal{I}_{(t+1)}}-\mathbf{X}^{\mathcal{I}_{(t)}}\right\|_{\text{F}}^2 \right]\nonumber\\
				& \quad +(1+\mu_5^{-1}) \mathbb{E}\left[\left\|\mathbf{X}^{\mathcal{I}_{(t)}}-\widehat{\mathbf{X}}^{\mathcal{I}_{(t)}}\right\|_{\text{F}}^2\right], \label{Ya3} \\
				&\mathbb{E}\left[\left\|\mathbf{Z}^{\mathcal{I}_{(t+1)}}-\overline{\mathbf{X}}^{\mathcal{I}_{(t)}}\right\|_{\text{F}}^2\right]
				\le (1+\mu_6) \mathbb{E} \left[ \left\|\mathbf{Z}^{\mathcal{I}_{(t+1)}}-\mathbf{X}^{\mathcal{I}_{(t)}}\right\|_{\text{F}}^2\right]\nonumber\\
				&\quad+(1+\mu_6^{-1})\mathbb{E} \left[ \left \|\mathbf{X}^{\mathcal{I}_{(t)}}-\overline{\mathbf{X}}^{\mathcal{I}_{(t)}}\right\|_{\text{F}}^2\right]. \label{Ya4}
			\end{align}
			Finally,  substituting (\ref{Ya2})--(\ref{Ya4}) into (\ref{Ya}) and utilizing the bounds in Lemma \ref{lee.2}, it can be seen that (\ref{L4}) is proved. 
		\end{proof}

	To proceed, we state the following recursive bounds that incorporate the consensus and compression errors.
	\begin{lemma}\label{lee.5}
		Under Assumptions~\ref{ass5.Un}, \ref{ass2.YY} and \ref{ass4.YJ}, for any $t\ge1$, let the sequence $\{\mathbf{X}^{{t}}\}$ be generated by Algorithm \ref{alg.1} with $\psi \in (\frac{3}{4r},\frac{1}{r}]$ and $\gamma$ satisfying (\ref{Gma}). Then, when the stepsizes are set to $\eta_{ t}=s_t=\frac{b}{\sqrt{t+a}}$ with $a>\frac{4\mathcal{H}}{3\zeta _1\left( \gamma \right)}$  and $\eta_{ t}=s_t=\frac{b}{{t+a}}$ with $a>\frac{8\mathcal{H}}{3\zeta _1\left( \gamma \right)}$, one respectively can obtain the following two inequalities 
		\begin{align}\label{L51}
			&\mathbb{E}\left[\left\|\mathbf{X}^{{t}}-\overline{\mathbf{X}}^{{t}}\right\|_{\text{F}}^2\right]+\mathbb{E}\left[\left\|\mathbf{X}^{{t}}-\widehat{\mathbf{X}}^{{t}}\right\|_{\text{F}}^2\right] \leq C_{1} \eta_{ t}^2
		\end{align}
		where the positive constant $C_{1}$ is given by
		\begin{equation}\label{equ.Ce_def}
			C_{1} = \frac{16\zeta_2(\gamma)\mathcal{H}^2 + 16N\psi r\delta^2}{\zeta_1(\gamma)^2} + 8N\mathcal{D}^2\mathcal{H}^2.
		\end{equation}
		and
		\begin{align}\label{L52}
			&\mathbb{E}\left[\left\|\mathbf{X}^{{t}}-\overline{\mathbf{X}}^{{t}}\right\|_{\text{F}}^2\right]+\mathbb{E}\left[\left\|\mathbf{X}^{{t}}-\widehat{\mathbf{X}}^{{t}}\right\|_{\text{F}}^2\right] \leq 	C_{2}\eta_t^2,
		\end{align}
		where the positive constant $C_{2}$ is given by
		\begin{equation}\label{equ.Ce2_def}
			C_{2} := \frac{32\zeta_2(\gamma)\mathcal{H}^2 + 32N\psi r\delta^2}{\zeta_1(\gamma)^2} + 16N\mathcal{D}^2\mathcal{H}^2,
		\end{equation}		\vspace{-3mm}
		where 		\vspace{-3mm}
		\begin{align}\label{Zt1}
			\zeta_1(\gamma)&:=\min \left\{\frac{\varphi}{4}-2\left( 1+\frac{4}{\kappa } \right) \gamma\alpha ^2-\left( 1-\frac{\varphi}{4} \right) 4\gamma \alpha,\right.\nonumber\\
			&\left.\frac{3\kappa \gamma}{2}-\frac{9\kappa ^2\gamma ^2}{16}-2\left( 1-\psi r\varphi \right) \left( 1+\frac{4}{\varphi} \right) \gamma ^2\alpha ^2\right\},
		\end{align}		\vspace{-3mm}
		and 		\vspace{-3mm}
		\begin{align}\label{Zt2}
			\zeta_2(\gamma) &:= \left( 1+\frac{4}{\kappa \gamma} \right) \Big[ 2\gamma^2\alpha^2 + 3\left( 1+\frac{\kappa \gamma}{4} \right) \nonumber\\
			&\quad \times \big( (1-\kappa \gamma +\gamma)^2 + \gamma^2 + (1-\kappa \gamma)^2 \big) \Big] N\mathcal{D}^2 \nonumber\\
			&\quad + \left( 1-\psi r\varphi \right) \left( 1+\frac{4}{\varphi} \right) \nonumber\\
			&\quad \times \Big[ \left( 1+\frac{\varphi}{4} \right) (1+\gamma \alpha)^2 + 2\gamma^2\alpha^2 \Big] N\mathcal{D}^2.
		\end{align}
	\end{lemma}
			\vspace{-3mm}
	\begin{proof}
		We first define $E_{{t}}=\mathbb{E} [\|\mathbf{X}^{{t}}-\overline{\mathbf{X}}^{{t}}\|_{\text{F}}^2] +	\mathbb{E}  {[\|\mathbf{X}^{{t}}-\widehat{\mathbf{X}}^{{t}}\|_{\text{F}}^2] }$. Building on the results in Lemmas~\ref{lee.3} and \ref{lee.4}, one can obtain the following recursion among epochs
		\begin{align}\label{CC0}
			& E_{\mathcal{I}_{(t+1)}} \nonumber\\&= \mathbb{E} \left[\left\|\mathbf{X}^{\mathcal{I}_{(t+1)}}-\overline{\mathbf{X}}^{\mathcal{I}_{(t+1)}}\right\|_{\text{F}}^2\right] + \mathbb{E} \left[\left\|\mathbf{X}^{\mathcal{I}_{(t+1)}}-\widehat{\mathbf{X}}^{\mathcal{I}_{(t+1)}}\right\|_{\text{F}}^2\right] \nonumber\\
			&\leq (\omega_1+\omega_4)\mathbb{E}\left[ \left\|\widehat{\mathbf{X}}^{\mathcal{I}_{(t)}}-\mathbf{X}^{\mathcal{I}_{(t)}}\right\|_{\text{F}}^2\right ]\nonumber\\
			&\quad+(\omega_2+\omega_5)\mathbb{E}\left[\left\|\mathbf{X}^{\mathcal{I}_{(t)}}-\overline{\mathbf{X}}^{\mathcal{I}_{(t)}}\right\|_{\text{F}}^2\right ]\nonumber\\
			&\quad +(\omega_3+\omega_6)N \mathcal{D}^2\mathcal{H}^2\eta_{\mathcal{I}_{(t)}}^2 + Ns_{\mathcal{I}_{(t)}}^2\psi r\delta^2.
		\end{align}
				\vspace{-3mm}

		Noting the arbitrariness of $\mu_i$, without loss of generality, let $\mu_1=\frac{4}{\kappa \gamma}$, $\mu_2=1$, $\mu_3=\frac{4}{\kappa \gamma}$, $\mu_4=\frac{\varphi}{4}$, $\mu_5=\frac{4}{\varphi}$ and $\mu_6=1$. 
		Specifically, simple calculations lead to 
		\begin{align}\label{CC1}
			&\omega_1+\omega_4\nonumber\\
			=&\left( 1+\mu _1 \right) \left( 1+\mu _2 \right) \gamma ^2\alpha^2\nonumber\\
			&+\left( 1-\psi r \varphi \right) \left( 1+\mu _4 \right) \left( 1+\mu _{5}^{-1} \right) \left( 1+\gamma \alpha \right) ^2\nonumber\\
			=&2\left( 1+\frac{4}{\kappa \gamma} \right) \gamma ^2\alpha^2+\left( 1-\psi r \varphi \right) \left( 1+\frac{\varphi}{4} \right) ^2\left( 1+\gamma \alpha \right) ^2\nonumber\\
			\le& 2\left( 1+\frac{4}{\kappa \gamma} \right) \gamma ^2\alpha^2+\left( 1-\frac{\varphi}{4} \right) \left( 1+4\gamma \alpha \right)\nonumber\\
			\le&1-\frac{\varphi}{4}+2\left( 1+\frac{4}{\kappa} \right) \gamma\alpha ^2+\left( 1-\frac{\varphi}{4} \right) 4\gamma \alpha \nonumber\\
			\le&1-\zeta_1(\gamma), 
		\end{align}
		where the first inequality is established by $\left( 1- \psi r\varphi \right) \left( 1+\frac{\varphi}{4} \right)^2 \le 1-\frac{1}{4}{\varphi}$, $\left( 1+\gamma \alpha \right)^2 \le 1+4\gamma \alpha $ and $\gamma^2<\gamma$. 
		
		Similarly, it can be concluded that 
		\begin{align}\label{CC2}
			&\omega_2+\omega_5\nonumber\\
			=&\left( 1+\mu_{1}^{-1} \right) \left( 1+\mu_{3}^{-1} \right) \left( 1- \kappa\gamma \right) ^2\nonumber\\
			&+\left( 1-\psi r\varphi \right) \left( 1+\mu_{4}^{-1} \right) \left( 1+\mu_{6}^{-1} \right) \gamma ^2\alpha ^2\nonumber\\
			=&\left( 1+\frac{\kappa \gamma}{4} \right) ^2\left( 1-\kappa \gamma \right) ^2+2\left( 1-\psi r\varphi \right) \left( 1+\frac{4}{\varphi} \right) \gamma ^2\alpha ^2\nonumber\\
			\le& 1-\frac{3\kappa \gamma}{2}+\frac{9\kappa ^2\gamma ^2}{16}+2\left( 1-\psi r\varphi \right) \left( 1+\frac{4}{\varphi} \right) \gamma ^2\alpha ^2\nonumber\\
			\le&1-\zeta_1(\gamma), 
		\end{align}
		where the first inequality is due to fact that $\left( 1+\frac{\kappa \gamma}{4} \right) \left( 1-\kappa \gamma \right) \le 1-\frac{3\kappa \gamma}{4}$.
		
		Subsequently, substitute (\ref{CC1}) and (\ref{CC2}) into (\ref{CC0}), we have
		\begin{align}\label{PT0}
			E_{\mathcal{I}_{(t+1)}}\le& (1-\zeta_1(\gamma))	E_{\mathcal{I}_{(t)}}+\zeta_2(\gamma)\mathcal{H}^2\eta_{\mathcal{I}_{(t)}}^2\nonumber\\
			&+ N\psi r\delta^2s_{\mathcal{I}_{(t)}}^2,
		\end{align}
		where $0 <\zeta_1(\gamma )<1$.
		
		In the sequel, we first deduce the upper bounds of ${E}_{\mathcal{I}_{(t)}}$ by induction under two different step size settings, and subsequently derive the upper bounds of ${E}_{t}$ based on these results.
		
		For $\mathcal{I}_{(1)}=1$, we have $E_{\mathcal{I}_{(1)}}=\mathbb{E} \left[\left\|\mathbf{X}^{\mathcal{I}_{(1)}}-\overline{\mathbf{X}}^{\mathcal{I}_{(1)}}\right\|_{\text{F}}^2\right] +	\mathbb{E}  {\left[\left\|\mathbf{X}^{\mathcal{I}_{(1)}}-\widehat{\mathbf{X}}^{\mathcal{I}_{(1)}}\right\|_{\text{F}}^2\right] }=0<\frac{4\zeta _2\left( \gamma \right) \mathcal{H}^2\eta _{\mathcal{I}_{\left( 1 \right)}}^{2}+4N\psi r\delta^2s_{\mathcal{I}_{\left( 1 \right)}}^2}{\zeta _1\left( \gamma \right) ^2}$. Subsequently, for some $\mathcal{I}_{(t)}$, suppose 
		\begin{align}\label{PT1}
			E_{\mathcal{I}_{(t)}}\le \frac{4\zeta _2\left( \gamma \right)\mathcal{H}^2 \eta _{\mathcal{I}_{\left( t \right)}}^{2}+4N\psi r\delta^2 s_{\mathcal{I}_{\left( t \right)}}^2}{\zeta _1\left( \gamma \right) ^2}.
		\end{align}
		Then for $\mathcal{I}_{(t+1)}$, by (\ref{PT0}) and (\ref{PT1}), we have 
		\begin{align}\label{PT2}
			&E_{\mathcal{I}_{(t+1)}}\nonumber\\
			&\le \left( 1-\zeta _1\left( \gamma \right) \right) \frac{4\zeta _2\left( \gamma \right) \mathcal{H}^2\eta _{\mathcal{I}_{\left( t \right)}}^{2}+4N\psi r\delta^2 s_{\mathcal{I}_{\left( t \right)}} ^2}{\zeta _1\left( \gamma \right) ^2}\nonumber\\
			&\quad+\zeta _2\left( \gamma \right)\mathcal{H}^2 \eta _{\mathcal{I}_{\left( t \right)}}^{2}+N\psi r\delta^2 s_{\mathcal{I}_{\left( t \right)}} ^2\nonumber\\
			&\le \frac{4}{\zeta _1\left( \gamma \right) ^2}\left( 1-\frac{3\zeta _1\left( \gamma \right)}{4} \right)  \left( \zeta _2\left( \gamma \right) \mathcal{H}^2\eta _{\mathcal{I}_{\left( t \right)}}^{2}+N\psi r\delta^2 s_{\mathcal{I}_{\left( t \right)}}  ^2 \right). 
		\end{align}
		Meanwhile, by (\ref{jinchou}), note that
		\begin{align}\label{PT5}
			&E_{t}\nonumber\\
			=&\mathbb{E}\left[\left\| \left(\mathbf{X}^{\mathcal{I}_{(t)0}}-\overline{\mathbf{X}}^{\mathcal{I}_{(t)0}}\right) -  \sum_{t'=\mathcal{I}_{(t)0}}^{\mathcal{I}_{(t+1)}-1} \left(\boldsymbol{\vartheta}_{t'} - \frac{1}{N}\mathbf{1}_N\mathbf{1}_N^\top\boldsymbol{\vartheta}_{t'} \right) \right\|_{\text{F}}^2  \right]\nonumber\\
			&+\mathbb{E}\left[\left\| \left(\mathbf{X}^{\mathcal{I}_{(t)0}}-\widehat{\mathbf{X}}^{\mathcal{I}_{(t)0}}\right) - \sum_{t'=\mathcal{I}_{(t)0}}^{\mathcal{I}_{(t+1)}-1} \boldsymbol{\vartheta}_{t'} \right\|_{\text{F}}^2  \right]\nonumber\\
			\le& 2E_{\mathcal{I}_{(t)0}} + 4N \mathcal{D}^2 \mathcal{H}^2\eta_{\mathcal{I}_{(t)0}}^2,
		\end{align}
		where $\mathcal{I}_{(t)0}\in \mathcal{I}_{T}$ denotes the most recent communication step before or equal to $t$ and $\vartheta_{t'} = \eta_{t'}\left (x_i^{t'}-\widetilde{\mathcal{T}}_i(x_i^{t'},\xi_{i}^{t'})\right )$.

		Next, we will discuss two different cases based on the varying values of the step sizes.
		
		\textbf{Case I.} For the case of $
		\eta _{t}=s_{t}=\frac{b}{\sqrt{t+a}}
		$, we have
		\begin{align*}
			\frac{\eta _{\mathcal{I}_{\left( t+1 \right)}}^{2}}{\eta _{\mathcal{I}_{\left( t \right)}}^{2}}=\frac{\mathcal{I}_{\left( t \right)}+a}{\mathcal{I}_{\left( t+1 \right)}+a}\ge 1-\frac{\mathcal{H}}{\mathcal{I}_{\left( t \right)}+a+\mathcal{H}}\ge 1-\frac{\mathcal{H}}{a}.
		\end{align*}
		Combined with $
		a\ge \frac{4\mathcal{H}}{3\zeta _1\left( \gamma \right)}
		$, one can further obtain
		\begin{align}\label{PT3}
			\left( 1-\frac{3\zeta _1\left( \gamma \right)}{4} \right) \eta _{\mathcal{I}_{\left( t \right)}}^{2}\le \left( 1-\frac{\mathcal{H}}{a} \right) \eta _{\mathcal{I}_{\left( t \right)}}^{2}\le \eta _{\mathcal{I}_{\left( t+1 \right)}}^{2}.
		\end{align}
		Then, in light of (\ref{PT2}) and (\ref{PT3}), it can be verified that
		\begin{align}\label{PT4}
			E_{\mathcal{I}_{(t+1)}}\le \frac{4}{\zeta _1\left( \gamma \right) ^2}\left( \zeta _2\left( \gamma \right) \mathcal{H}^2\eta _{\mathcal{I}_{\left( t+1 \right)}}^{2}+N\psi r\delta^2 s_{\mathcal{I}_{\left( t+1 \right)}}^2 \right). 
		\end{align}
		Further, by $\eta _{ t }=s_{t}=\frac{b}{\sqrt{t+a}}$, we have 
		\begin{align}\label{PT6}
			\frac{\eta _{\mathcal{I}_{\left( t \right) 0}}^{2}}{\eta _{t}^{2}}=\frac{t+a}{\mathcal{I}_{\left( t \right) 0}+a}&\le \frac{\mathcal{I}_{\left( t \right) 0}+\mathcal{H}+a}{\mathcal{I}_{\left( t \right) 0}+a}\nonumber\\
			&\le 1+\frac{\mathcal{H}}{\mathcal{I}_{\left( t \right) 0}+a}\le 2.
		\end{align}
		Evidently, combining (\ref{PT4}) (\ref{PT5}) and (\ref{PT6}), gives rise to
		\begin{align}\label{C1}
			E_{t+1}\le& 2 E_{\mathcal{I}_{(t+1)0}}+4N  \mathcal{D}^2 \mathcal{H}^2\eta_{\mathcal{I}_{(t+1)0}}^2\nonumber\\
			\le&
			\frac{16\zeta _2\left( \gamma \right) \mathcal{H}^2\eta _{t+1}^{2}+16N\psi r\delta^2s_{t+1}^2}{\zeta _1\left( \gamma \right) ^2}\nonumber\\
			&+8N \mathcal{D}^2 \mathcal{H}^2\eta_{t+1}^2.
		\end{align}
		
		\textbf{Case II.} For the case of $
		\eta _{t}=s_{{t}}=\frac{b}{t+a}$, we arrive at
		\begin{align*}
			\frac{\eta _{\mathcal{I}_{\left( t+1 \right)}}^{2}}{\eta _{\mathcal{I}_{\left( t \right)}}^{2}}=\left( \frac{\mathcal{I}_{\left( t \right)}+a}{\mathcal{I}_{\left( t+1 \right)}+a} \right) ^2\ge 1-\frac{2\mathcal{H}}{\mathcal{I}_{\left( t \right)}+a+\mathcal{H}}\ge 1-\frac{2\mathcal{H}}{a}.
		\end{align*}
		Then, since $
		a\ge \frac{8\mathcal{H}}{3\zeta _1\left( \gamma \right)}
		$, similar analysis in Case I can deduce
		\begin{align}\label{C2}
			E_{t+1}\le& 2E_{\mathcal{I}_{(t+1)0}}+4N  \mathcal{D}^2 \mathcal{H}^2\eta_{\mathcal{I}_{(t+1)0}}^2\nonumber\\
			\le&
			\frac{32\zeta _2\left( \gamma \right) \mathcal{H}^2\eta _{t+1}^{2}+32N\psi r\delta^2 s_{t+1}^2}{\zeta _1\left( \gamma \right) ^2}\nonumber\\
			&+16N \mathcal{D}^2 \mathcal{H}^2\eta_{t+1}^2.
		\end{align}
		
		As a result, the proof is completed by (\ref{C1}) and (\ref{C2}).
	\end{proof} \vspace{-8mm}

	\vspace{-3mm}
	\subsection{Proof of Theorem \ref{Theo1}}\label{FE}\vspace{-3mm}
		\begin{proof}
		(i) First, we bound the cumulative consensus error. Noting that $\|\mathbf{X}^t - \overline{\mathbf{X}}^t\|_{\text{F}}^2 = \sum_{i=1}^N \|x_i^t - \bar{x}^t\|^2$ and $s_t = \eta_t$, it follows from \eqref{L51} that
				\vspace{-3mm}
		\begin{align}
			\sum_{i=1}^N \mathbb{E}\left[\|x_i^t - \bar{x}^t\|^2\right] \leq C_{1} \eta_t^2.\label{eq:consensus_compact}
		\end{align}
		By summing \eqref{eq:consensus_compact} over $t=1$ to $T$ with $\eta_t = \frac{b}{\sqrt{t+a}}$, we obtain
				\vspace{-3mm}
		\begin{align*}
			\sum_{t=1}^T \sum_{i=1}^N \mathbb{E}\left[\|x_i^t - \bar{x}^t\|^2\right] \leq C_{1} \sum_{t=1}^T \frac{b^2}{t+a} 
			\leq C_{1} b^2 \ln\left(1 + \frac{T}{a}\right). 
		\end{align*}
				\vspace{-3mm}
				
		(ii) Recall that $\bar{x}^{t+1} = \bar{x}^t - \eta_t \bar{d}^t$ with $\bar{d}^t=\frac{1}{N}\sum_{i=1}^{N}(x_{i}^t-\widetilde{\mathcal{T}}_i(x_{i}^t,\xi_{i}^t))$. Since the global surrogate function $\mathcal{G}(x) = \frac{1}{N} \sum_{i=1}^N \mathcal{G}_i(x)$ is $(1+L)$-smooth by Lemma~\ref{lemm.aa}, it gives rise to
				\vspace{-3mm}
		\begin{align}
			&\mathcal{G}(\bar{x}^{t+1}) - \mathcal{G}(\bar{x}^t) \nonumber \\
			&\leq - \eta_t \langle \nabla \mathcal{G}(\bar{x}^t), \bar{d}^t \rangle + \frac{(1+L) \eta_t^2}{2} \|\bar{d}^t\|^2. \label{eq:descent_initial}
		\end{align}  				\vspace{-8mm}
		
		Note that $\nabla \mathcal{G}_i(x) = x - \mathcal{T}_i(x)$, we have 		\vspace{-3mm}
		\begin{align}
			&- \eta_t \langle \nabla \mathcal{G}(\bar{x}^t), \bar{d}^t \rangle \nonumber \\
			&= - \frac{\eta_t}{N} \sum_{i=1}^N \langle \bar{x}^t - \mathcal{T}(\bar{x}^t), x_i^t - \mathcal{T}_i(x_i^t) \rangle \nonumber \\
			&\quad - \frac{\eta_t}{N} \sum_{i=1}^N \langle \bar{x}^t - \mathcal{T}(\bar{x}^t), \mathcal{T}_i(x_i^t) - \widetilde{\mathcal{T}}_i(x_i^t, \xi_i^t) \rangle.
			\label{eq:inner_product_decoupled}
		\end{align}  				\vspace{-8mm}

		Given $0 \le P < 1$, we decompose the first term on the right hand side of \eqref{eq:inner_product_decoupled} into two parts by introducing the weight $\frac{2P}{P+1}$ and its complement $\frac{1-P}{P+1}$ as follows:		\vspace{-3mm}
		\begin{align}\label{LLTT53}
			&	- \frac{\eta_t}{N} \sum_{i=1}^N \langle \bar{x}^t - \mathcal{T}(\bar{x}^t), x_i^t - \mathcal{T}_i(x_i^t) \rangle\nonumber\\
			&= -\frac{\eta_t}{N} \frac{2P}{P+1} \sum_{i=1}^{N} \langle \bar{x}^t - \mathcal{T}(\bar{x}^t), x_i^t - \mathcal{T}_i(x_i^t) \rangle \nonumber\\
			&\quad- \eta_t \frac{1-P}{P+1} \langle \bar{x}^t - \mathcal{T}(\bar{x}^t), \frac{1}{N} \sum_{i=1}^{N} (x_i^t - \mathcal{T}_i(x_i^t)) \rangle.
		\end{align}				\vspace{-8mm}
		
		Bounding the first part of \eqref{LLTT53} yields
		\begin{align}
			&-\frac{\eta_t}{N} \frac{2P}{P+1} \sum_{i=1}^{N} \langle \bar{x}^t - \mathcal{T}(\bar{x}^t), x_i^t - \mathcal{T}_i(x_i^t) \rangle \nonumber\\
			&\leq -\frac{\eta_t P}{P+1} \|\bar{x}^t - \mathcal{T}(\bar{x}^t)\|^2 - \frac{\eta_t P}{N(P+1)} \sum_{i=1}^{N} \|x_i^t - \mathcal{T}_i(x_i^t)\|^2 \nonumber\\
			&\quad + \frac{2\eta_t P}{P+1} \zeta^2 + \frac{2\eta_t P (1+L)^2}{N(P+1)} \sum_{i=1}^{N} \|x_i^t - \bar{x}^t\|^2.
		\end{align} 				\vspace{-8mm}
		
		Bounding the second part of \eqref{LLTT53} yields
		\begin{align}
			&- \eta_t \frac{1-P}{P+1} \langle \bar{x}^t - \mathcal{T}(\bar{x}^t), \frac{1}{N} \sum_{i=1}^{N} (x_i^t - \mathcal{T}_i(x_i^t)) \rangle \nonumber\\
			&\leq -\frac{\eta_t(1-P)}{2(P+1)} \|\bar{x}^t - \mathcal{T}(\bar{x}^t)\|^2 \nonumber\\
			&\quad+ \frac{\eta_t(1-P)}{2(P+1)} \left\| \frac{1}{N} \sum_{i=1}^{N} \left( (\bar{x}^t - x_i^t) + (\mathcal{T}_i(x_i^t) - \mathcal{T}_i(\bar{x}^t)) \right) \right\|^2\nonumber\\
			&\leq -\frac{\eta_t(1-P)}{2(P+1)} \|\bar{x}^t - \mathcal{T}(\bar{x}^t)\|^2 \nonumber\\
			&\quad+ \frac{\eta_t(1-P)(1+L)^2}{2N(P+1)} \sum_{i=1}^{N} \|x_i^t - \bar{x}^t\|^2.
		\end{align} 				\vspace{-8mm}
		
		Combining the above two bounds, we obtain
		\begin{align}
			&- \eta_t \langle \nabla \mathcal{G}(\bar{x}^t), \bar{d}^t \rangle \nonumber \\
			& \leq -\frac{\eta_t}{2} \|\bar{x}^t - \mathcal{T}(\bar{x}^t)\|^2 - \frac{\eta_t P}{N(P+1)} \sum_{i=1}^{N} \|x_i^t - \mathcal{T}_i(x_i^t)\|^2 \nonumber\\
			&\quad+ \frac{2\eta_t P}{P+1} \zeta^2 + \frac{\eta_t (3P+1)(1+L)^2}{2N(P+1)} \sum_{i=1}^{N} \|x_i^t - \bar{x}^t\|^2.
		\end{align} 				\vspace{-8mm}
		
		Define $b_i^t := \mathcal{T}_i(x_i^t) - \mathbb{E}[ \widetilde{\mathcal{T}}_i(x_i^t, \xi_i^t)]$. Invoking Young's inequality with $\rho > 0$ and Assumption \ref{ass2.YY}, we have \vspace{-3mm}
		\begin{align}
			&- \frac{\eta_t}{N} \sum_{i=1}^N \mathbb{E} \big[ \langle \bar{x}^t - \mathcal{T}(\bar{x}^t), b_i^t \rangle \big] \nonumber \\
			&\leq \frac{\eta_t \rho}{2} \|\bar{x}^t - \mathcal{T}(\bar{x}^t)\|^2 + \frac{\eta_t}{2\rho N} \sum_{i=1}^N \mathbb{E}\left[\|b_i^t\|^2\right] \nonumber \\
			&\leq \frac{\eta_t \rho}{2} \|\bar{x}^t - \mathcal{T}(\bar{x}^t)\|^2 + \frac{\eta_t}{2\rho} \beta^2 \nonumber\\
			&\quad + \frac{\eta_t P}{2\rho N} \sum_{i=1}^N \mathbb{E}\left[\|x_i^t - \mathcal{T}_i(x_i^t)\|^2\right]. \label{eq:bias_bound_final}
		\end{align} 				\vspace{-8mm}
		
		Furthermore, the direction $\bar{d}^t$ can be rewritten as
		\begin{align}
			\bar{d}^t &= \big(\bar{x}^t - \mathcal{T}(\bar{x}^t)\big) + \frac{1}{N} \sum_{i=1}^N \big[ (Id-\mathcal{T}_i)(x_i^t) - (Id-\mathcal{T}_i)(\bar{x}^t) \big] \nonumber \\
			&\quad + \frac{1}{N} \sum_{i=1}^N \big( \mathcal{T}_i(x_i^t) - \widetilde{\mathcal{T}}_i(x_i^t, \xi_i^t) \big). \label{eq:delta_decomposition}
		\end{align} 
		Applying the inequality $\|a+b+c\|^2 \leq 3\|a\|^2 + 3\|b\|^2 + 3\|c\|^2$ to \eqref{eq:delta_decomposition}, while leveraging Assumptions \ref{ass1.LL} and \ref{ass2.YY}, we have
		\begin{align}
			\mathbb{E} \big[ \|\bar{d}^t\|^2 \big] 
			&\leq 3 \|\bar{x}^t - \mathcal{T}(\bar{x}^t)\|^2 + \frac{3(1+L)^2}{N} \sum_{i=1}^N \|x_i^t - \bar{x}^t\|^2 \nonumber \\
			&\quad + \frac{3}{N} \sum_{i=1}^N \left( \sigma^2 + M\|x_i^t - \mathcal{T}_i(x_i^t)\|^2 \right). \label{eq:deterministic_final}
		\end{align} 	\vspace{-8mm}
		
		Furthermore, by applying the inequality $\|\sum_{i=1}^4 a_i\|^2 \leq 4\sum_{i=1}^4 \|a_i\|^2$, we have
		\begin{align}\label{LLL60}
			&\|x_i^t - \mathcal{T}_i(x_i^t)\|^2 \nonumber\\
			=& \|(x_i^t - \bar{x}^t) + (\bar{x}^t - \mathcal{T}(\bar{x}^t)) \nonumber\\
			&+ (\mathcal{T}(\bar{x}^t) - \mathcal{T}_i(\bar{x}^t)) + (\mathcal{T}_i(\bar{x}^t) - \mathcal{T}_i(x_i^t))\|^2 \nonumber\\
			\leq& 4 \|x_i^t - \bar{x}^t\|^2 + 4 \|\bar{x}^t - \mathcal{T}(\bar{x}^t)\|^2 \nonumber\\
			&+ 4 \|\mathcal{T}(\bar{x}^t) - \mathcal{T}_i(\bar{x}^t)\|^2 + 4 \|\mathcal{T}_i(\bar{x}^t) - \mathcal{T}_i(x_i^t)\|^2.
		\end{align} 	\vspace{-8mm}

		Based on (\ref{LLL60}), one can derive that
		\begin{align}\label{eq:local_operator_bound}
			\frac{1}{N} \sum_{i=1}^{N} &\|x_i^t - \mathcal{T}_i(x_i^t)\|^2 \leq 4 \|\bar{x}^t - \mathcal{T}(\bar{x}^t)\|^2 + 4 \zeta^2 \nonumber\\
			&+ \frac{4(1+L^2)}{N} \sum_{i=1}^{N} \|x_i^t - \bar{x}^t\|^2,
		\end{align}
		where the inequality is obtained by Assumptions~\ref{ass1.LL} and \ref{ass.heteroge}.
		
		Substituting \eqref{eq:local_operator_bound} into \eqref{eq:deterministic_final} and grouping the corresponding terms, we obtain
		\begin{align}\label{eq:d_bound_refined}
			&\mathbb{E} \big[ \|\bar{d}^t\|^2 \big] \leq \; 3(1+4M) \|\bar{x}^t - \mathcal{T}(\bar{x}^t)\|^2 + 3(\sigma^2 + 4M\zeta^2) \nonumber \\
			&+ \frac{3 \big( (1+L)^2 + 4M(1+L^2) \big)}{N} \sum_{i=1}^N \|x_i^t - \bar{x}^t\|^2.
		\end{align}
		Substituting \eqref{eq:inner_product_decoupled}--\eqref{eq:bias_bound_final} and (\ref{eq:d_bound_refined}) into \eqref{eq:descent_initial} with the carefully chosen $\rho = \frac{P+1}{2}$, we have
		\begin{align}
			&\mathbb{E} \big[ \mathcal{G}(\bar{x}^{t+1}) - \mathcal{G}(x^*) \big] - \mathbb{E} \big[ \mathcal{G}(\bar{x}^t) - \mathcal{G}(x^*) \big] \nonumber \\
			&\leq - \eta_t \varpi_{1,t} \mathbb{E} \left[ \|\bar{x}^t - \mathcal{T}(\bar{x}^t)\|^2 \right] + \varpi_{2,t} \sum_{i=1}^N \mathbb{E} \left[ \|x_i^t - \bar{x}^t\|^2 \right] \nonumber \\
			&\quad + \frac{\eta_t}{P+1} \beta^2 + \frac{2\eta_t P}{P+1} \zeta^2 + \frac{3(1+L)\eta_t^2}{2} (\sigma^2 + 4M\zeta^2), \label{eq:final_descent}
		\end{align}
		where the parameters $\varpi_{1,t}, \varpi_{2,t}$ are defined as 
		\begin{align*}
			\varpi_{1,t} := &\frac{1-P}{4} - \frac{3(1+L)(1+4M)\eta_t}{2}, \\
			\varpi_{2,t} :=& \frac{\eta_t (3P+1)(1+L)^2}{2N(P+1)} \nonumber\\
			&+ \frac{3(1+L)\eta_t^2 \big( (1+L)^2 + 4M(1+L^2) \big)}{2N}.
		\end{align*} \vspace{-8mm}

		By selecting a non-increasing sequence $\{\eta_{t}\}$ and noting that $P<1$, we can select an initialization $\eta_{0} < \frac{1-P}{6(1+L)(1+4M)}$, which guarantees that $\varpi_{1,t} \geq \underline{\varpi}_1 > 0$ for all $t\ge 0$, where $\underline{\varpi}_1 := \frac{1-P}{4} - \frac{3(1+L)(1+4M)\eta_0}{2}$. 
		
		Rearranging the terms in \eqref{eq:final_descent} and noting $\mathbb{E} \big[ \mathcal{G}(\bar{x}^{T+1}) - \mathcal{G}(x^*) \big] \ge  0$, we have
		\begin{align}\label{eq:sum_gradient}
			&\sum_{t=0}^T \eta_t \mathbb{E} \big[ \|\bar{x}^t - \mathcal{T}(\bar{x}^t)\|^2 \big] \nonumber \\
			&\leq \frac{1}{\underline{\varpi}_1} \mathbb{E} \big[ \mathcal{G}(\bar{x}^0) - \mathcal{G}(x^*) \big] 
			+ \frac{1}{\underline{\varpi}_1} \sum_{t=0}^T \varpi_{2,t} \sum_{i=1}^N \mathbb{E} \big[ \|x_i^t - \bar{x}^t\|^2 \big] \nonumber \\
			&\quad + \frac{1}{\underline{\varpi}_1} \sum_{t=0}^T \Bigg( \frac{\eta_t}{P+1} \beta^2 + \frac{2\eta_t P}{P+1} \zeta^2 \nonumber \\
			&\qquad\qquad\qquad\quad + \frac{3(1+L)\eta_t^2}{2} (\sigma^2 + 4M\zeta^2) \Bigg).
		\end{align}
						\vspace{-8mm}

		Utilizing the $(1+L)$-smoothness of $Id-\mathcal{T}_i$, we have
		\begin{align}
			&\frac{1}{N}\sum_{i=1}^N \mathbb{E} \left[ \| x_i^t - \mathcal{T}(x_i^t) \|^2 \right] \nonumber \\
			&= \frac{1}{N}\sum_{i=1}^N \mathbb{E} \Bigg[ \bigg\| \frac{1}{N}\sum_{j=1}^N \big( (Id-\mathcal{T}_j)(x_i^t) - (Id-\mathcal{T}_j)(\bar{x}^t) \big) \nonumber \\
			&\quad + \big( \bar{x}^t - \mathcal{T}(\bar{x}^t) \big) \bigg\|^2 \Bigg] \nonumber \\
			&\leq \frac{2(1+L)^2}{N} \sum_{i=1}^N \mathbb{E} \left[ \|x_i^t - \bar{x}^t\|^2 \right] + 2 \mathbb{E} \left[ \|\bar{x}^t - \mathcal{T}(\bar{x}^t)\|^2 \right]. \label{eq:final_local_operator_bound_prep}
		\end{align}
						\vspace{-8mm}

		Let $\overline{\varpi}_2 = \max_t \varpi_{2,t}$. Summing \eqref{eq:final_local_operator_bound_prep} from $t=1$ to $T$, and substituting $\sum_{t=1}^T \mathbb{E}[\|\bar{x}^t - \mathcal{T}(\bar{x}^t)\|^2] \leq \frac{1}{\eta_T} \sum_{t=0}^T \eta_t \mathbb{E}[\|\bar{x}^t - \mathcal{T}(\bar{x}^t)\|^2]$ along with \eqref{eq:sum_gradient}, it follows that
		\begin{align}
			&\sum_{t=1}^T \frac{1}{N} \sum_{i=1}^N \mathbb{E} \big[ \| x_i^t - \mathcal{T}(x_i^t) \|^2 \big] \nonumber \\
			&\leq \left( \frac{2 (1+L)^2}{N} + \frac{2 \overline{\varpi}_2}{\underline{\varpi}_1 \eta_T} \right) \sum_{t=1}^T \sum_{i=1}^N \mathbb{E} \big[ \|x_i^t - \bar{x}^t\|^2 \big] \nonumber \\
			&\quad + \frac{2}{\underline{\varpi}_1 \eta_T} \sum_{t=0}^T \Bigg( \frac{\eta_t}{P+1} \beta^2 + \frac{2\eta_t P}{P+1} \zeta^2 \nonumber \\
			&\qquad\qquad\qquad\quad + \frac{3(1+L)\eta_t^2}{2} (\sigma^2 + 4M\zeta^2) \Bigg)\nonumber \\
			&\quad + \frac{2}{\underline{\varpi}_1 \eta_T} \mathbb{E} \big[ \mathcal{G}(\bar{x}^0) - \mathcal{G}(x^*) \big]. \label{eq:final_local_operator_bound}
		\end{align}
						\vspace{-8mm}

		The accumulated step-size bounds are derived as
		\begin{align}
			\sum_{t=0}^T \eta_t &\leq \eta_0 + \int_0^T \frac{b}{\sqrt{t+a}} dt \leq \eta_0 + 2b\sqrt{T+a}, \label{eq:sum_eta} \\
			\sum_{t=0}^T \eta_t^2 &\leq \eta_0^2 + b^2 \ln\left(1 + \frac{T}{a}\right). \label{eq:sum_eta_sq}
		\end{align} \vspace{-8mm}
		
		Dividing both sides of \eqref{eq:final_local_operator_bound} by $T$ and substituting the bounds \eqref{eq:sum_eta} and \eqref{eq:sum_eta_sq} along with $\eta_T = \frac{b}{\sqrt{T+a}}$, we obtain
		\begin{align}
			&\frac{1}{T} \sum_{t=1}^T \frac{1}{N} \sum_{i=1}^N \mathbb{E} \left[ \| x_i^t - \mathcal{T}(x_i^t) \|^2 \right] \nonumber \\
			&\leq \frac{1}{T} \left( \frac{2(1+L)^2}{N} + \frac{2 \overline{\varpi}_2 \sqrt{T+a}}{\underline{\varpi}_1 b} \right) C_{1} b^2 \ln\left(1 + \frac{T}{a}\right) \nonumber \\
			&\quad + \frac{2 \sqrt{T+a}}{\underline{\varpi}_1 b T} \left( \frac{\beta^2}{P+1} + \frac{2P\zeta^2}{P+1} \right) \sum_{t=0}^T \eta_t \nonumber \\
			&\quad + \frac{2 \sqrt{T+a}}{\underline{\varpi}_1 b T} \cdot \frac{3(1+L)(\sigma^2 + 4M\zeta^2)}{2} \sum_{t=0}^T \eta_t^2 \nonumber \\
			&\quad + \frac{2 \sqrt{T+a}}{\underline{\varpi}_1 b T} \mathbb{E} \big[ \mathcal{G}(\bar{x}^0) - \mathcal{G}(x^*) \big]. \label{eq:final_rate_expansion}
		\end{align}				\vspace{-8mm}
		
		Evaluating the asymptotic order with respect to $T$ for the right-hand side of \eqref{eq:final_rate_expansion} yields
		\begin{align}
			&\frac{1}{T} \sum_{t=1}^T \frac{1}{N} \sum_{i=1}^N \mathbb{E} \left[ \| x_i^t - \mathcal{T}(x_i^t) \|^2 \right] \nonumber\\
			&\le \mathcal{O}\left( \frac{C_{1}\ln T}{\sqrt{T}} \right) + \mathcal{O}\left( \frac{\beta^2}{P+1} + \frac{2P\zeta^2}{P+1} \right). \label{eq:final_O_notation}
		\end{align}  \vspace{-8mm}
		
		This completes the  proof.
	\end{proof}

	\begin{lemma}\label{lee.6}
		Consider a nonnegative sequence $\{\varPsi_t\}$ satisfying
		\begin{align}\label{eq:psi_dynamics} 
			\varPsi_{t+1} \leq \left( 1-\frac{r_1}{t+a} \right) \varPsi_t  + \frac{r_2}{(t+a)^2} + \frac{r_3}{t+a},
		\end{align}
		where $r_1,r_2, r_3 > 0$. If $r_1 \geq 1$ and $a > r_1 $, then for all $t\geq 0$,
		\begin{align}\label{Lem5}
			\varPsi_{t} \leq \frac{D_1 \ln(t+a) + D_2}{t-1+a} + 2r_3,
		\end{align}
		with $D_1 = 2r_2$ and $D_2 = a\varPsi_0 + r_2\left(1 + \frac{2}{a}\right) + r_3$.
	\end{lemma}
		\begin{proof}
			For $t \in \{0, 1\}$, \eqref{Lem5} holds trivially. In the following, we consider the case of $t\ge 2$. 
			Define the state transition $\Phi(t, s) \triangleq \prod_{k=s}^{t} \big(1 - \frac{r_1}{k+a}\big)$ for $t \geq s$. Using $1-x \leq e^{-x}$, we bound the decay rate
			\begin{align}\label{eq:phi_bound}
				\Phi(t, s) &\leq \exp\left( -r_1 \sum_{k=s}^{t} \frac{1}{k+a} \right) \leq \exp\left( -r_1 \int_{s}^{t} \frac{1}{\tau+a} d\tau \right)\nonumber\\
				& = \left( \frac{s+a}{t+a} \right)^{r_1}.
			\end{align} \vspace{-8mm}
			
			Unrolling \eqref{eq:psi_dynamics} with the convention $\Phi(t-1, t) \equiv 1$ yields
			\begin{align}\label{eq:psi_unrolled}
				\varPsi_t \leq& \Phi(t-1, 0) \varPsi_0 + r_2 \sum_{k=0}^{t-1} \frac{\Phi(t-1, k+1)}{(k+a)^2} \nonumber\\&+ r_3 \sum_{k=0}^{t-1} \frac{\Phi(t-1, k+1)}{k+a}.
			\end{align} \vspace{-8mm}
			
			Since $r_1 \geq 1$, we have
			\begin{align}\label{eq:term1}
				\Phi(t-1, 0) \varPsi_0 \leq \left( \frac{a}{t-1+a} \right)^{r_1} \varPsi_0 \leq \frac{a \varPsi_0}{t-1+a}.
			\end{align} \vspace{-8mm}
			
			For $k \leq t-2$, since $a > 1$, we have $\frac{k+1+a}{k+a} \leq 2$. Thus,
			\begin{align}\label{eq:algebraic_relaxation}
				(k+1+a)^{r_1} \leq 2(k+a) (t-1+a)^{r_1-1}.
			\end{align} \vspace{-8mm}
			
			Isolating the $k=t-1$ term and applying \eqref{eq:algebraic_relaxation}, one can see that 
			\begin{align}\label{eq:term2}
				&r_3 \sum_{k=0}^{t-1} \frac{\Phi(t-1, k+1)}{k+a} \nonumber\\
				&= r_3 \sum_{k=0}^{t-2} \frac{\Phi(t-1, k+1)}{k+a} + \frac{r_3}{t-1+a} \nonumber \\
				&\leq \frac{r_3}{(t-1+a)^{r_1}} \sum_{k=0}^{t-2} \frac{2(k+a)(t-1+a)^{r_1-1}}{k+a} + \frac{r_3}{t-1+a} \nonumber \\
				&\leq 2r_3 + \frac{r_3}{t-1+a}.
			\end{align} \vspace{-8mm}
			
			Similarly, applying \eqref{eq:algebraic_relaxation} and $\sum_{k=0}^{t-2} \frac{1}{k+a} \leq \ln(t+a) + \frac{1}{a}$, it can be seen that
			\begin{align}\label{eq:term3}
				&r_2 \sum_{k=0}^{t-1} \frac{\Phi(t-1, k+1)}{(k+a)^2} \nonumber\\
				&\leq \frac{r_2}{(t-1+a)^{r_1}} \sum_{k=0}^{t-2} \frac{2(k+a)(t-1+a)^{r_1-1}}{(k+a)^2} + \frac{r_2}{(t-1+a)^2} \nonumber \\
				&\leq \frac{2r_2 \big( \ln(t+a) + \frac{1}{a} \big)}{t-1+a} + \frac{r_2}{t-1+a}.
			\end{align} \vspace{-8mm}

			Substituting \eqref{eq:term1}--\eqref{eq:term3} into \eqref{eq:psi_unrolled} yields
			\begin{align*}
				\varPsi_t &\leq \frac{a \varPsi_0}{t-1+a} + 2r_3 + \frac{r_3}{t-1+a} \nonumber\\
				&	\quad+ \frac{2r_2 \ln(t+a) + \frac{2r_2}{a}}{t-1+a} + \frac{r_2}{t-1+a} \\
				&= \frac{2r_2 \ln(t+a) + \left( a\varPsi_0 + r_2\left(1 + \frac{2}{a}\right) + r_3 \right)}{t-1+a} + 2r_3.
			\end{align*}
			This matches \eqref{Lem5} with $D_1 = 2r_2$ and $D_2 = a\varPsi_0 + r_2\left(1 + \frac{2}{a}\right) + r_3$.
		\end{proof}

	\vspace{-3mm}
	\subsection{Proof of Theorem \ref{Theo2}}\label{FF}\vspace{-3mm}
	\begin{proof}
		(i) Recalling $\|\mathbf{X}^t - \overline{\mathbf{X}}^t\|_{\text{F}}^2=\sum_{i=1}^N \|x_i^t - \bar{x}^t\|^2  $ and invoking \eqref{L52}, we have
		\begin{align}
			\sum_{i=1}^{N}\mathbb{E}[\|x_{i}^t-\bar{x}^t\|^2]\le C_2 \eta_{t}^2 = \frac{C_2 b^2}{(t+a)^2}.
		\end{align}
						\vspace{-8mm}
		
		(ii) Let $\varPsi_t := \mathbb{E} \big[ \mathcal{G}(\bar{x}^t) - \mathcal{G}^*\big]$ denote the expected surrogate gap evaluated at the network average state. Under the contractive condition ($L < 1$), invoking Lemma~\ref{lemm.aa}, we have $\|\bar{x}^t - \mathcal{T}(\bar{x}^t)\|^2 \geq 2(1-L)\big( \mathcal{G}(\bar{x}^t)-\mathcal{G}^* \big)$.

		Substituting this into \eqref{eq:final_descent} yields  				\vspace{-5mm}
		\begin{align}
			\Psi_{t+1} \leq \;& \big( 1 - 2\eta_t \varpi_{1,t} (1-L) \big) \Psi_t + \varpi_{2,t} \sum_{i=1}^N \mathbb{E} \left[ \|x_i^t - \bar{x}^t\|^2 \right] \nonumber \\
			&+ \frac{\eta_t}{P+1} \beta^2 + \frac{2\eta_t P}{P+1} \zeta^2 + \frac{3(1+L)\eta_t^2}{2} (\sigma^2 + 4M\zeta^2). \label{eq:psi_recurrence}
		\end{align}	\vspace{-8mm}
						
		Selecting $a \geq \frac{12(1+L)(1+4M)b}{1-P}$ guarantees that $\varpi_{1,t} \geq \frac{1-P}{8}$ for all $t \geq 0$, which consequently yields
		\begin{align}\label{L77}
			1 - 2\eta_t \varpi_{1,t} (1-L) \leq 1 - \frac{(1-L)(1-P)b}{4(t+a)}.
		\end{align} 	\vspace{-8mm}

		Substituting \eqref{L77} and the definition of $\varpi_{2,t}$ into \eqref{eq:psi_recurrence}, and applying the bounds $a \geq \frac{12(1+L)(1+4M)b}{1-P}$ and $\frac{1}{(t+a)^k} \le \frac{1}{(t+a)^2}$ for $k \ge 3$, we have
		\begin{align}
			\mathbb{E} \big[ \mathcal{G}(\bar{x}^{t+1}) - \mathcal{G}^* \big] \leq \;& \left( 1 - \frac{r_1}{t+a} \right) \mathbb{E} \big[ \mathcal{G}(\bar{x}^t) - \mathcal{G}^* \big]\nonumber\\
			& + \frac{r_2}{(t+a)^2} + \frac{r_3}{t+a}, \label{eq:mapped_recurrence}
		\end{align}
		where the constants are defined as
		\begin{align*}
			r_1 &:= \frac{(1-L)(1-P)b}{4}, \\
			r_2 &:= \frac{3(1+L)b^2 (\sigma^2 + 4M\zeta^2)}{2} + \frac{b^3(3P+1)(1+L)^2 C_2}{2N(P+1)} \\
			&\quad + \frac{3(1+L) C_2 b^4 \big( (1+L)^2 + 4M(1+L^2) \big)}{2N}, \\
			r_3 &:= b \left( \frac{\beta^2 + 2P\zeta^2}{P+1} \right).
		\end{align*} \vspace{-8mm}

		To satisfy the premise of Lemma \ref{lee.6}, we require $b > \frac{4}{(1-L)(1-P)}$. In this case, applying Lemma \ref{lee.6} directly to \eqref{eq:mapped_recurrence} yields
		\begin{align}\label{eq:lemma6_applied}
			\mathbb{E} \big[ \mathcal{G}(\bar{x}^t) - \mathcal{G}^* \big] \leq \frac{D_1 \ln(t+a) + D_2}{t-1+a} + {2r_3},
		\end{align}
		where $D_1 = 2r_2$ and $D_2 = a\varPsi_0 + r_2\left(1 + \frac{2}{a}\right) + r_3$.
		
		Since $Id-\mathcal{T}$ is $(1+L)$-Lipschitz continuous, we have \vspace{-3mm}
		\begin{align}\label{eq:local_gap_expansion}
			&\frac{1}{N} \sum_{i=1}^N \mathbb{E} \big[ \mathcal{G}(x_i^t) - \mathcal{G}^* \big] \nonumber \\
			&\leq \frac{1}{N} \sum_{i=1}^N \mathbb{E} \bigg[ \mathcal{G}(\bar{x}^t) - \mathcal{G}^* + \langle \bar{x}^t - \mathcal{T}(\bar{x}^t), x_i^t - \bar{x}^t \rangle\bigg]\nonumber\\
			&\quad   + \frac{1+L}{2N} \sum_{i=1}^N\mathbb{E} \bigg[ \|x_i^t - \bar{x}^t\|^2 \bigg] \nonumber \\
			&= \mathbb{E} \big[ \mathcal{G}(\bar{x}^t) - \mathcal{G}^* \big] + \frac{1+L}{2N} \sum_{i=1}^N \mathbb{E} \big[ \|x_i^t - \bar{x}^t\|^2 \big].
		\end{align}
				\vspace{-8mm}
		
		Substituting the consensus error bound \eqref{L52} and incorporating the established  bound \eqref{eq:lemma6_applied}, we deduce \vspace{-3mm}
		\begin{align}\label{eq:final_exact_bound}
			&\frac{1}{N} \sum_{i=1}^N \mathbb{E} \big[ \mathcal{G}(x_i^t) - \mathcal{G}^* \big] \nonumber\\
			&\leq \mathbb{E} \big[ \mathcal{G}(\bar{x}^t) - \mathcal{G}^* \big] + \frac{1+L}{2N} C_2 \eta_t^2 \nonumber \\
			&\leq \frac{D_1 \ln(t+a) + D_2}{t-1+a} + 2r_3 + \frac{(1+L)C_2 b^2}{2N(t+a)^2} \nonumber \\
			&\leq \frac{D_1 \ln(t+a) + \widetilde{D}_2}{t-1+a} + \frac{2b(\beta^2 + 2P\zeta^2)}{P+1},
		\end{align}
		where $\widetilde{D}_2 := D_2 + \frac{(1+L)C_2 b^2}{2N}$, and the last step utilizes $\frac{1}{(t+a)^2} \le \frac{1}{t-1+a}$ for $t \ge 1$ alongside substituting $r_3$.

		Further, it follows from Lemma~\ref{lemm.aa} that $\frac{1-L}{2}\|x_i^t - x^*\|^2 \leq \mathcal{G}(x_i^t) - \mathcal{G}^*$. Applying this to \eqref{eq:final_exact_bound} gives rise to \vspace{-3mm}
		\begin{align*}
			\frac{1}{N} \sum_{i=1}^N \mathbb{E} \big[ \|x_i^t - x^*\|^2 \big] \leq \;& \frac{2}{1-L} \left( \frac{D_1 \ln(t+a) + \widetilde{D}_2}{t-1+a} \right) \nonumber\\
			&+ \frac{4b(\beta^2 + 2P\zeta^2)}{(1-L)(P+1)}.
		\end{align*} \vspace{-8mm}
		
		This completes the  proof.
	\end{proof}
}

\bibliographystyle{elsarticle-harv}
\bibliography{points}

\end{document}